\documentclass[sn-mathphys,Numbered]{sn-jnl}

\usepackage{graphicx}%
\usepackage{multirow}%
\usepackage{amsmath,amssymb,amsfonts}%
\usepackage{amsthm}%
\usepackage{mathrsfs}%
\usepackage[title]{appendix}%
\usepackage{xcolor}%
\usepackage{textcomp}%
\usepackage{manyfoot}%
\usepackage{booktabs}%
\usepackage{algorithm}%
\usepackage{algorithmicx}%
\usepackage{algpseudocode}%
\usepackage{listings}%
\usepackage{natbib}

\def\bfx{{\bf x}}
\def\bfy{{\bf y}}

\def\det{{\rm det}}

\theoremstyle{thmstyleone}%
\newtheorem{theorem}{Theorem}
\newtheorem{lemma}[theorem]{Lemma}%

\theoremstyle{thmstyletwo}%
\newtheorem{example}{Example}%
\newtheorem{remark}{Remark}%

\theoremstyle{thmstylethree}%

\begin{document}

\title{ Multivariate Splines and Their Applications}
\author*[1]{\fnm{Ming-Jun} \sur{Lai}}\email{mjlai@uga.edu} 
\affil[1]{\orgdiv{Department of Mathematics}, \orgname{University of Georgia}, \orgaddress{\city{Athens}, \postcode{30602}, \state{Georgia}, \country{USA}}}

\abstract{This paper begins by reviewing numerous theoretical advancements in the field of multivariate splines, primarily contributed by Professor Larry L. Schumaker. These foundational results have paved the way for a wide range of applications and computational techniques. The paper then proceeds to highlight various practical applications of multivariate splines. These include scattered data fitting and interpolation, the construction of smooth curves and surfaces, and the numerical solutions of various partial differential equations, encompassing both linear and nonlinear PDEs. Beyond these conventional and well-established uses, the paper introduces a novel application of multivariate splines in function value denoising. This innovative approach facilitates the creation of LKB splines, which are instrumental in approximating high-dimensional functions and effectively circumventing the curse of dimensionality.} 


\maketitle

\section{Introduction}\label{sec1}

This paper commemorates the 80th birthday of Professor Larry L. Schumaker and is derived from a plenary talk given by the author at a conference organized by Professor M. Neamtu and colleagues from May 15 to 18, 2023. A key area of Professor Schumaker's research, multivariate splines, is the focus of this paper. While Professor Schumaker has an extensive range of publications, this paper will selectively delve into his significant contributions to multivariate splines, highlighting their impact and ongoing developments in the field.

Multivariate splines are piecewise polynomial functions defined over a triangulated domain in $\mathbb{R}^d$ (where $d\ge 2$) or over a spherical triangulation, such as that of the Earth. These finite-dimensional spaces are extremely versatile for various approximation tasks. For instance, to visualize a set of scattered data points in 2D or 3D, bivariate or trivariate splines are often employed to fit or interpolate the data over the given points. The resulting spline surfaces or isosurfaces provide insights into the data's behavior. Multivariate splines are also integral to the numerical solutions of partial differential equations, akin to the widely recognized finite element methods. Another classic use of these splines is in constructing interpolatory smooth curves and surfaces. Furthermore, they are applied in image processing tasks, such as image edge detection (refer to \cite{GL13}) using box spline functions, and in image deformation. An illustration of image deformation, computed using bivariate splines, is presented in Figure~\ref{Bday} and will be further explained in a subsequent section.

\begin{figure}[htpb]
\centering
\includegraphics[width = 0.8\textwidth]{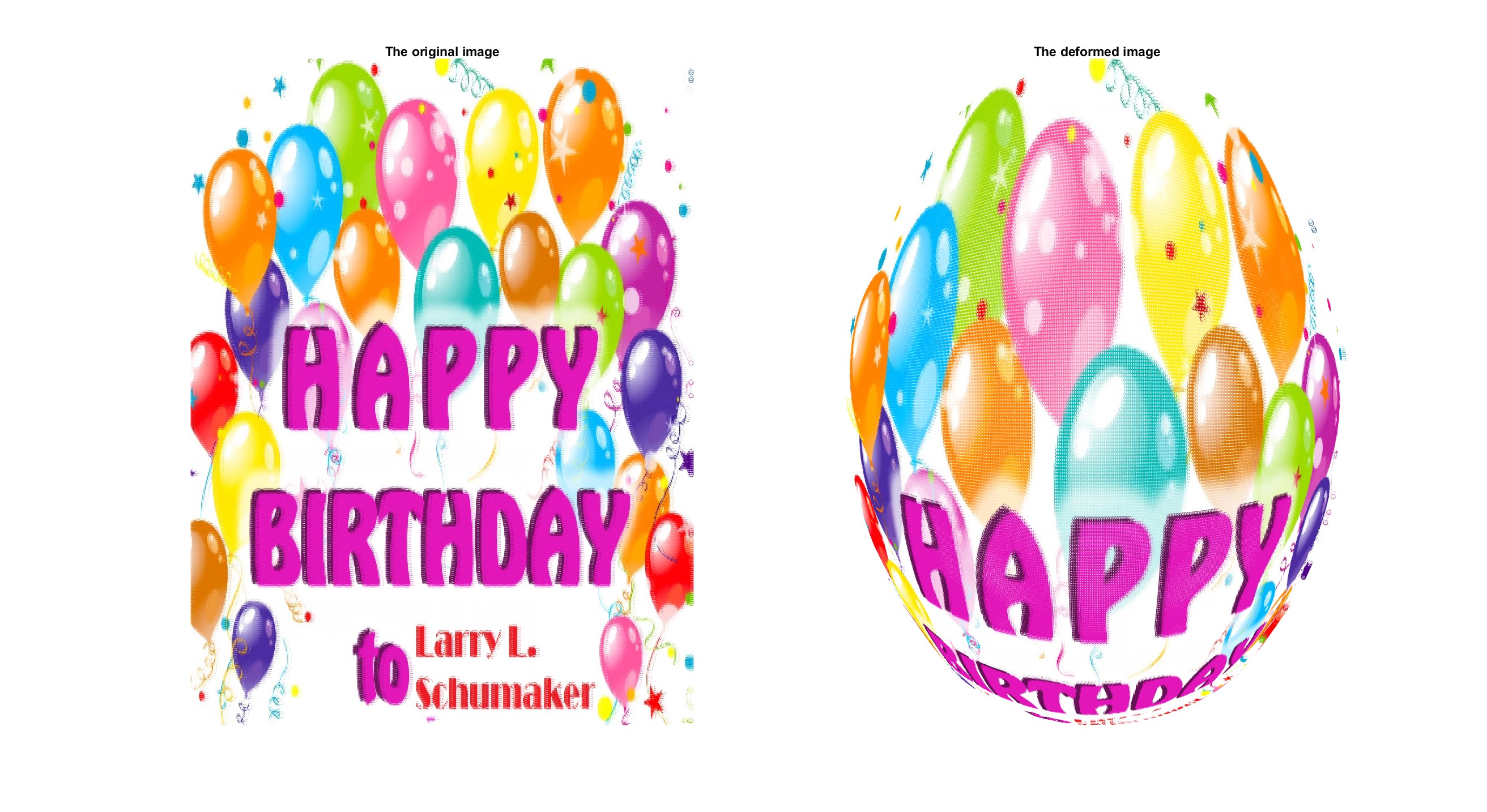}
\caption{A poster  (left) is deformed to the one on oval shaped domain (right) like a printed balloon}
\label{Bday}
\end{figure}

Multivariate splines have also been effectively employed in the denoising of functions. This is particularly notable in the context of the Kolmogorov superposition theorem, which is utilized for approximating multi-dimensional functions. The computation of the inherent K-inner functions within this theorem has traditionally been challenging, leading to a perception of its impracticality (as discussed in \cite{GP89}). However, the use of bivariate and trivariate splines for denoising Kolmogorov B-splines (KB splines) has revitalized the numerical approximation of high-dimensional continuous functions using the Kolmogorov superposition theorem, rendering it both feasible and meaningful.

The remainder of this paper is structured as follows: It will commence with a precise definition of multivariate splines, followed by a detailed discussion of three major research problems associated with these splines. A significant contribution in this field is the introduction of spherical splines in the 1990s by Professor Schumaker and his long-term collaborators. These spherical splines, encompassed as a specific category within multivariate splines, have enabled the approximation of geoscientific data. Beyond the theoretical exploration of multivariate splines, the paper will delve into their computation and diverse applications, illustrating several examples to demonstrate their extensive applicability.

\section{Definition of Multivariate Splines and Main Research Problems}
Multivariate splines are piecewise polynomial functions, 
usually defined on a triangulation in 2D, or a tetrahedral partition in 3D, 
or a spherical triangulation, or a  simplicial partition in $\mathbb{R}^n$, $n\ge 4$.  

Let ${\bf P}_d$ be the space of all polynomials of degree $d\ge 1$. 
Let $\Delta$ be a triangulation of a domain  $\Omega\subset \mathbb{R}^n, n\ge 2$. 
For integers $d\ge 1$, $-1\le r\le d$, let  
\begin{equation}
\label{splinedef}
S^r_d(\Delta) =\{ s\in C^r(\Omega), s|_t\in {\bf P}_d, t\in \Delta\}
\end{equation}
be the spline space of smoothness $r$ and degree $d$ over   $\Delta$. When $d=1$ and $r=0$, the spline space $S^0_1(\triangle)$ is the well-known 
finite element space.  For $r=-1$, $S^{-1}_d(\triangle)$ is the space of discontinuous splines over $\triangle$ which has been used for discontinuous Galerkin 
method for numerical solution of partial differential equations.    Similarly, when $r=0$, $S^0_d(\triangle)$ is used for continuous Galerkin method.

The primary objective of multivariate spline research is to accurately approximate both unknown and complex known functions in 2D, 3D, and spherical settings. Over the past 30 years, multivariate splines have evolved to become both flexible and convenient tools, proving to be immensely valuable in a variety of numerical applications, which will be elaborated on in subsequent sections.

To effectively employ multivariate splines for data fitting and approximation in computational settings, several fundamental questions arise. Firstly, what is the dimension of the spline space $S^r_d(\triangle)$ for any given values of $r$, $d$, and triangulation $\triangle$? Secondly, how proficiently can the spline space $S^r_d(\triangle)$ approximate both unknown and complex known functions? In other words, what are the approximation capabilities of $S^r_d(\triangle)$? The third pertinent question concerns the practical implementation of these spline functions in various applications.

Addressing these questions is challenging. For instance, a triangulation $\triangle$ can present significant complexity. As illustrated in \cite{WWLG20}, the locations of given data points and their boundary vertices are depicted in Figure~\ref{data1}. An associated triangulation, shown in Figure~\ref{tri1}, was employed to approximate the given data values and to estimate the density of mercury pollution across the entire area of the triangulated domain using spline functions in $S^1_5(\triangle)$.  

\begin{figure}[htpb]
\begin{center}
\includegraphics[width = 0.5\textwidth]{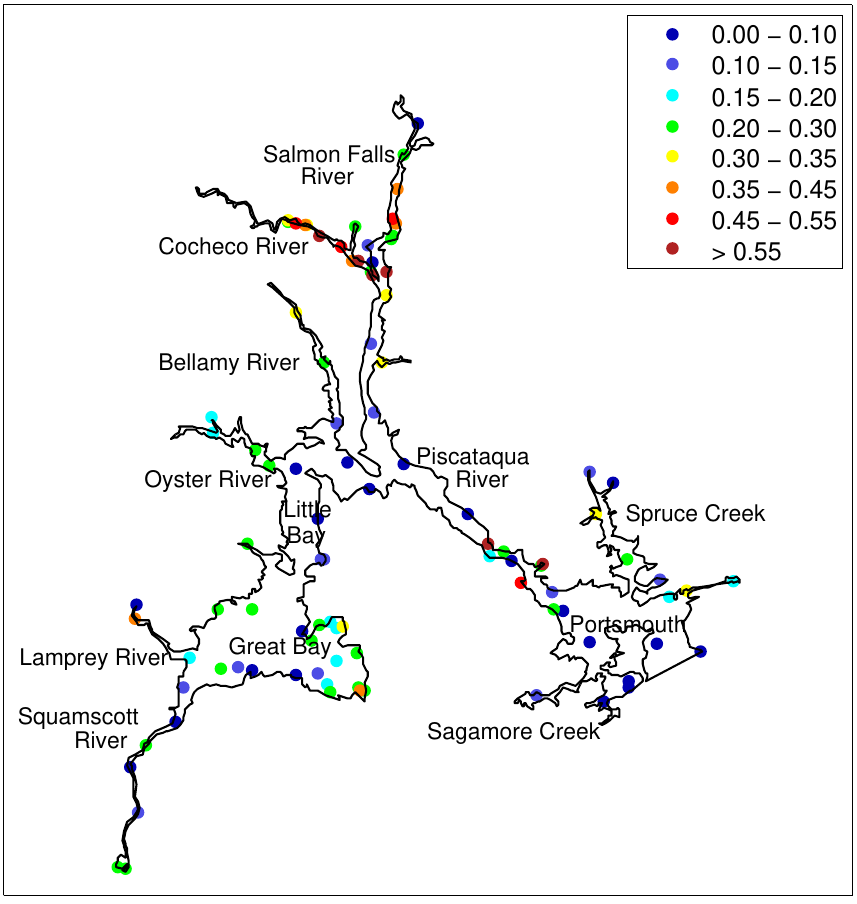}
\end{center}
\caption{A set of data locations associated with the mercury pollution as explained in \cite{WWLG20}}
\label{data1}
\end{figure}

\begin{figure}[htpb]
\includegraphics[width = 1.1\textwidth]{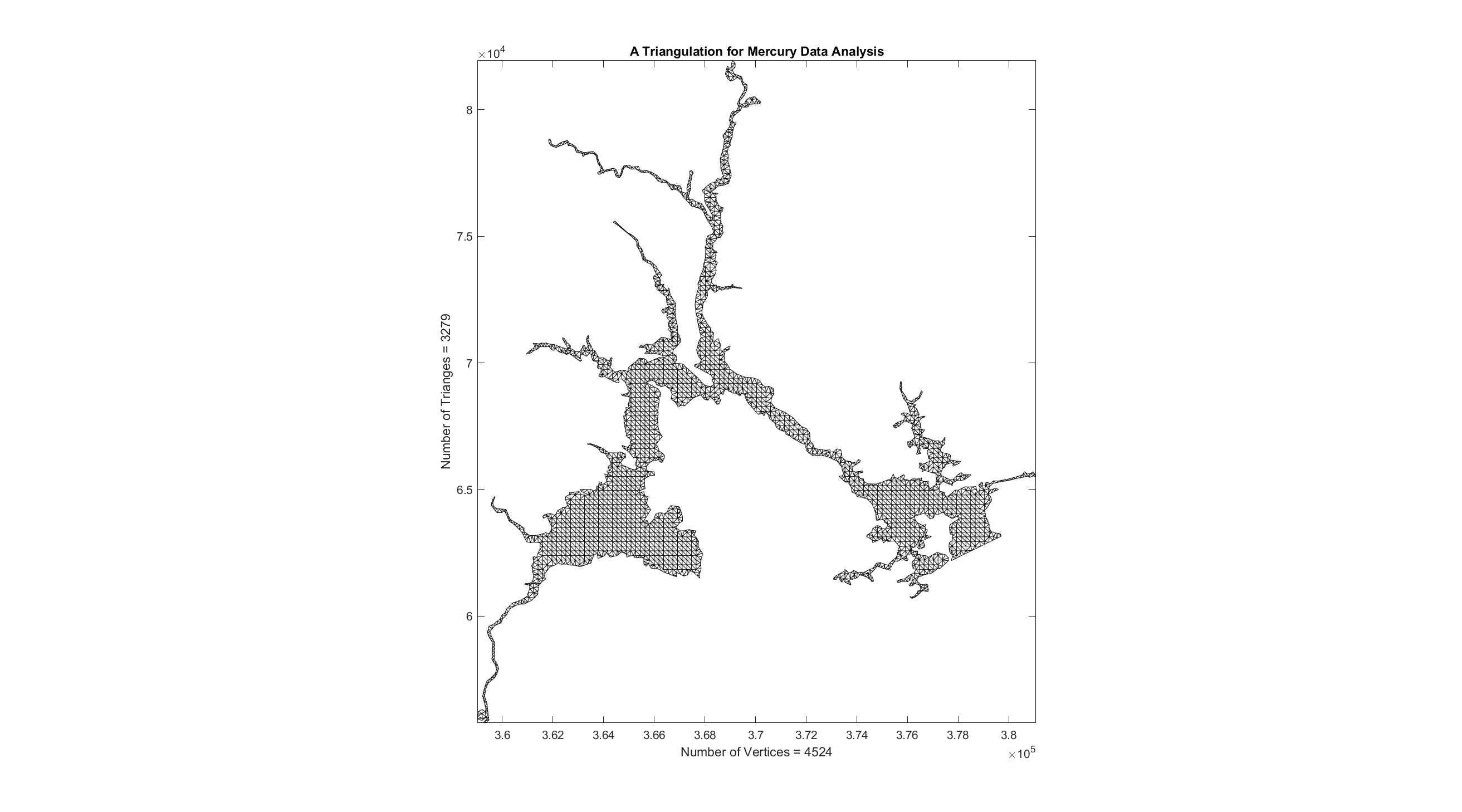}
\caption{An associated triangulation of the given data set in Figure~\ref{data1} } \label{tri1}
\end{figure}

From Figure~\ref{tri1}, one can see that the dimension $S^1_5(\triangle)$ is not easy to find or estimate.

\subsection{Dimension of Spline Spaces} 
During the 1970s and 1980s, Professor Schumaker devoted considerable effort to addressing the first key question in multivariate spline research. He published extensively on this subject, and the following is a summary of some notable findings derived from his work.  

\begin{theorem} [Schumaker, 1979\cite{S79} and Schumaker, 1984\cite{S84}]
\label{Sdim2}
Suppose that $\triangle$ is a triangulation of a given domain $\Omega\subset \mathbb{R}^2$. 
For any $0\le r\le d$, 
\begin{equation}
D+ \sum_{v\in {\cal V}_I}\sigma_v \le \dim (S^r_d(\triangle)) \le D + \sum_{v\in {\cal V}_I} \tilde{\sigma}_v
\end{equation}
where 
$$
D = {d+2\choose 2} + {d-r+1\choose 2}E_I - [{d+2\choose 2} - {r+2\choose 2}]V_I
$$
and $\sigma_v=\sum_{j=1}^{d-r} (r+j+1-jm_v)_+$ and $\tilde{\sigma}_v= \sum_{j=1}^{d-r} (r+j+1-j \tilde{m}_v)_+$. 
\end{theorem}

In this context, $m_v$ denotes the count of edges with distinct slopes connecting to a vertex $v$, and $\tilde{m}_v$ represents the count of such edges that have not been previously considered. This theorem demonstrates that the dimension of the spline space $S^r_d(\triangle)$ is influenced by the edge slopes in $\triangle$. Determining the exact dimension of a spline space for any given degree $d$ and smoothness $r\ge 1$ over any triangulation is notably challenging. However, the exact dimension of $S^r_d(\triangle)$ becomes more tractable when the smoothness $r$ is less than or equal to 0. 

The task of defining dimension formulas for trivariate spline spaces is much more complex, even though results akin to Theorem~\ref{Sdim2} have been achieved.

\begin{theorem} [Alfeld,  Schumaker, and Sirvent, 1992\cite{ASS92}]
\label{Sdim3}
Suppose that a tetrahedralization $\triangle$ is shellable. For any $0\le r\le d$, 
\begin{equation}
L(r,d) \le \dim (S^r_d(\triangle)) \le U(r,d)
\end{equation}
where 
$$
L(r,d) = {d+3\choose 3} +  n_0\ell_0+ n_1 \ell_1 + n_2\ell_2+ n_3\ell_3
$$
$$
U(r,d) = {d+3\choose 3} +  n_0 u_0+ n_1 u_1 + n_2u_2+ n_3u_3
$$
\end{theorem}

In this context, a tetrahedralization, denoted as $\triangle$, is defined as an assembly of tetrahedra with the property that any pair of tetrahedra either do not intersect or share one of the following common elements: a triangular face, an edge, or a vertex. A tetrahedralization $\triangle$ is considered shellable if there exists a specific sequencing of its tetrahedra, $T_1, \cdots, T_n$, such that one can construct $\triangle$ by starting with $\triangle_1=T_1$ and sequentially adding each $T_{k+1}$ to $\triangle_k$. In this process, each newly added tetrahedron $T_{k+1}$ must share one, two, or three triangular faces with the existing structure $\triangle_k$, for $k=1, 2, 3, \cdots, n-1$. It is important to note that not all tetrahedralizations $\triangle$ are shellable. A significant theoretical question in this field is whether any polyhedral domain --  defined as a simply connected domain with piecewise linear faces -- can always be represented by a shellable tetrahedralization $\triangle$ or not. This query remains unresolved and forms a key area of inquiry in geometric and computational studies.

In the realm of finite element methods, including Discontinuous Galerkin Methods (DG), the concept of Degrees of Freedom (DoF's) is a crucial metric for assessing computational efficiency. The dimension of spline spaces provides a precise and meaningful quantification of these DoF's, thereby offering a valuable tool for evaluating and optimizing computational processes. 

An open research problem is to determine the dimension of the following general spline space:
\begin{eqnarray}
\label{general}
S^{\bf r}_{\bf d}(\triangle) = \{ s:  & s|_{T_i} \in {\bf P}_{d_i}, i=1, \cdots, N, \cr
& s\in C^{r_j}(e_j), j=1, \cdots, M\},
\end{eqnarray}
where ${\bf d}=(d_1, \cdots, d_N)^\top$ is a vector of nonnegative integers 
with $N$ being the number of triangles in $\triangle = \bigcup_{i=1}^N T_i$,  
${\bf r}=(r_1, \cdots, r_M)$ is another vector of integers with $r_j\ge -1$  and 
$M$ being the number of interior edges of $\triangle$. That is, each spline function $s\in S^{\bf r}_{\bf d}(\triangle) $ is a piecewise polynomial function which is of degree $d_i$ over triangle $T_i$ and is 
$C^{r_j}$ across the jth interior edge. The dimension of this very general spline space was discussed by Lai and Schumaker in 
\cite{LS19} a few years ago.

\subsection{Spherical Spline Functions}
In the 1990s, Professor Schumaker, along with his colleagues, spearheaded the development of a novel type of splines on the unit ball's surface, known as spherical splines. These splines are characterized as piecewise spherical harmonics formed over spherical triangulations. This innovation was detailed in a trio of papers by Alfeld, Neamtu, and Schumaker in 1996, referenced as \cite{ANS96a}, \cite{ANS96b}, and \cite{ANS96c}. Subsequent research further explored the approximation properties of spherical splines, as seen in studies such as \cite{FS96}, \cite{BLS06}, \cite{BL11}, \cite{BL18}, among others.

Following the invention of spherical splines and the advancement in understanding their approximation properties, these tools found practical applications in real-world data handling, including interpolation, fitting, and approximation in spherical contexts. A notable application of spherical splines is in reconstructing the geopotential model using satellite measurements around the Earth, as discussed in \cite{LSBW09}.

\begin{example}[Geopotential Reconstruction (cf. \cite{LSBW09}]
A set of  the geopotential measurements from a German satellite was given. 
 See data locations and a triangulation of the Earth  in Figure~\ref{geodata}. 

\begin{figure}[ht]
\begin{tabular}{cc}
\includegraphics[width=0.4\textwidth]{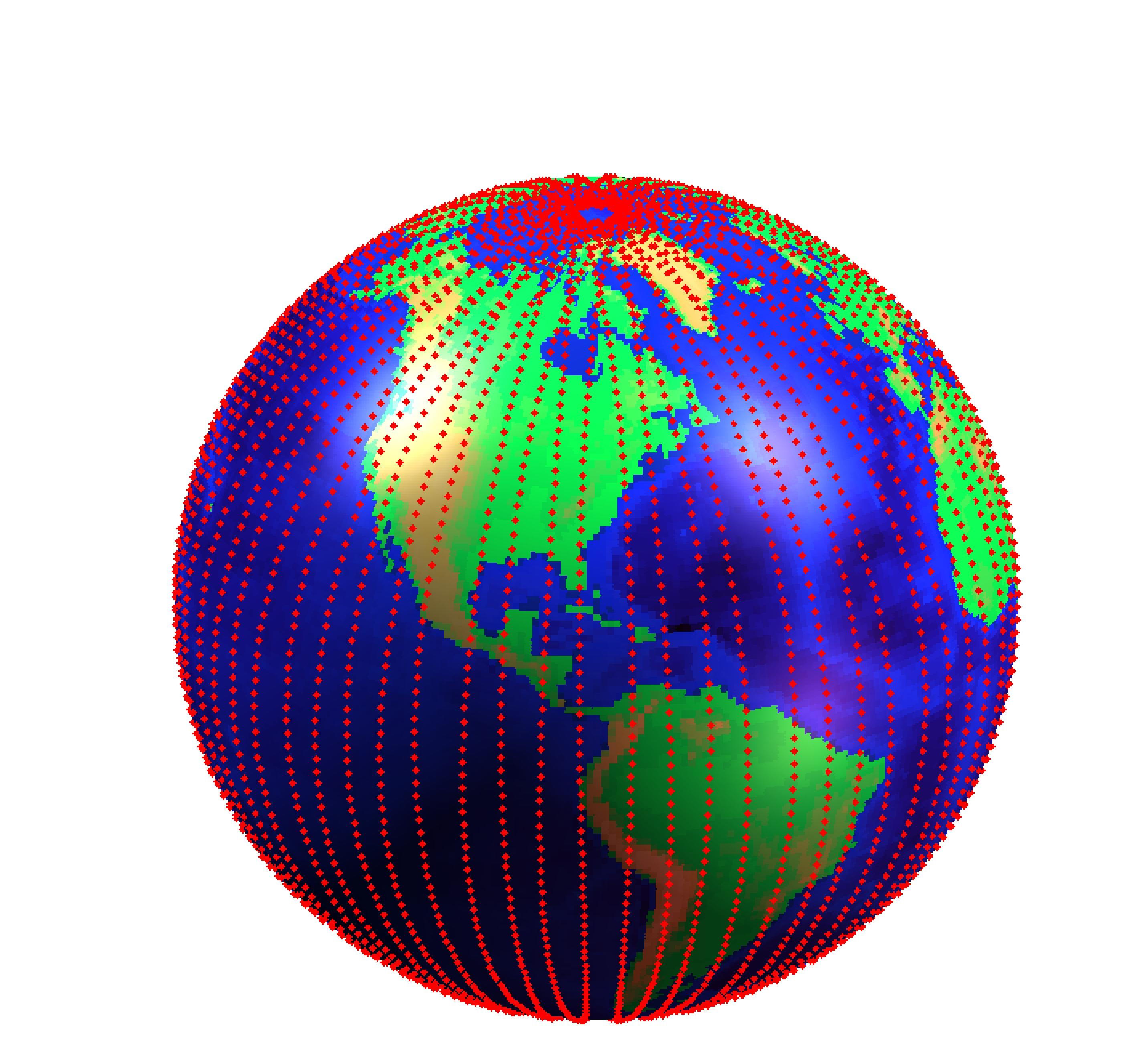}& 
\includegraphics[width=0.35\textwidth]{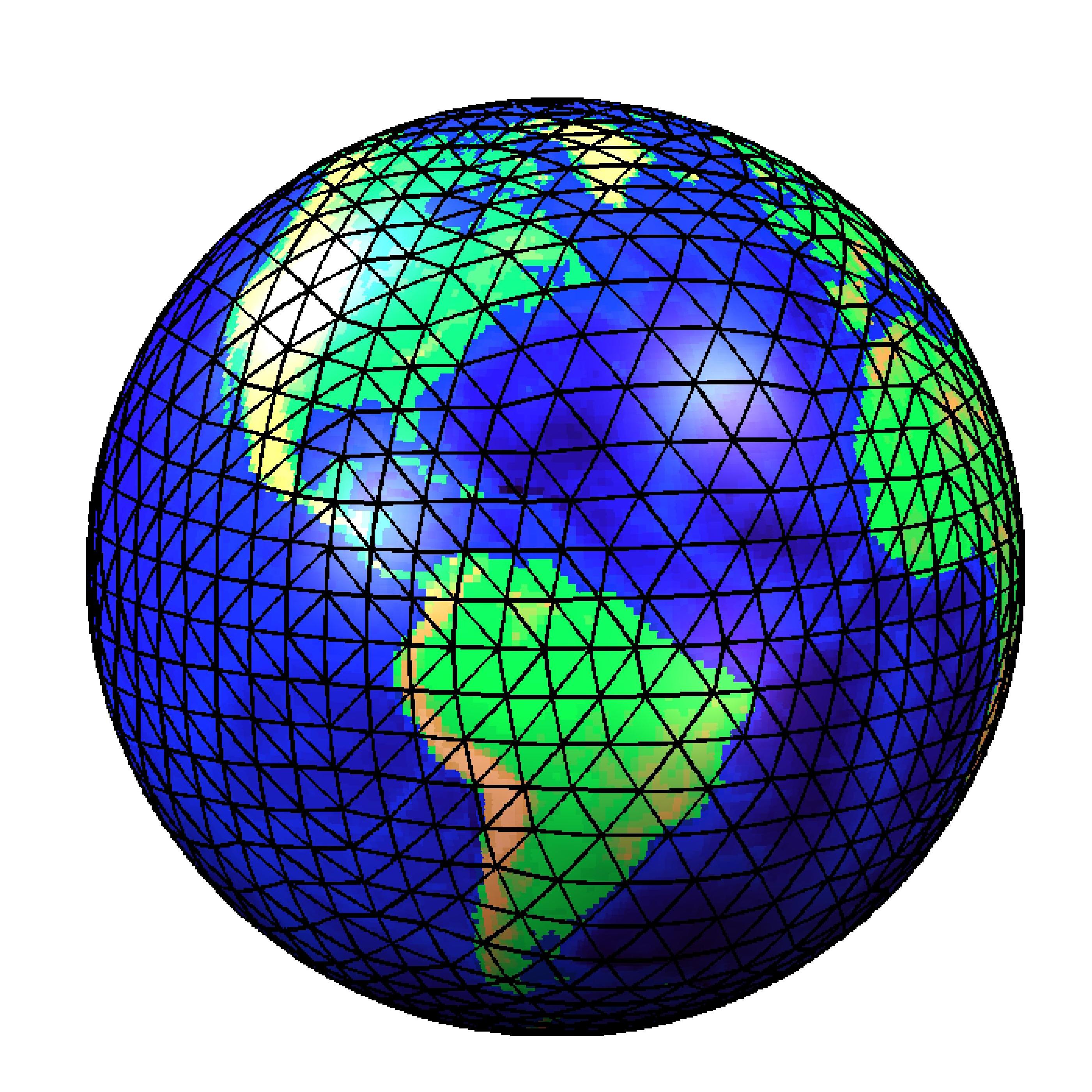}
\end{tabular}
\caption{Geo-potential measurement locations and a triangulation of the earth}
\label{geodata}
\end{figure}

Spherical splines of degree $5$ and smoothness $1$ over the triangulation as shown on the right of 
Figure~\ref{geodata} were used to fit the geopotential values over 
the  data locations as shown on the left of Figure~\ref{geodata} sufficiently 
accurately as demonstrated in Figure~\ref{geofit}. 
\begin{figure}[ht]
\begin{tabular}{cc}
\includegraphics[width=0.4\textwidth]{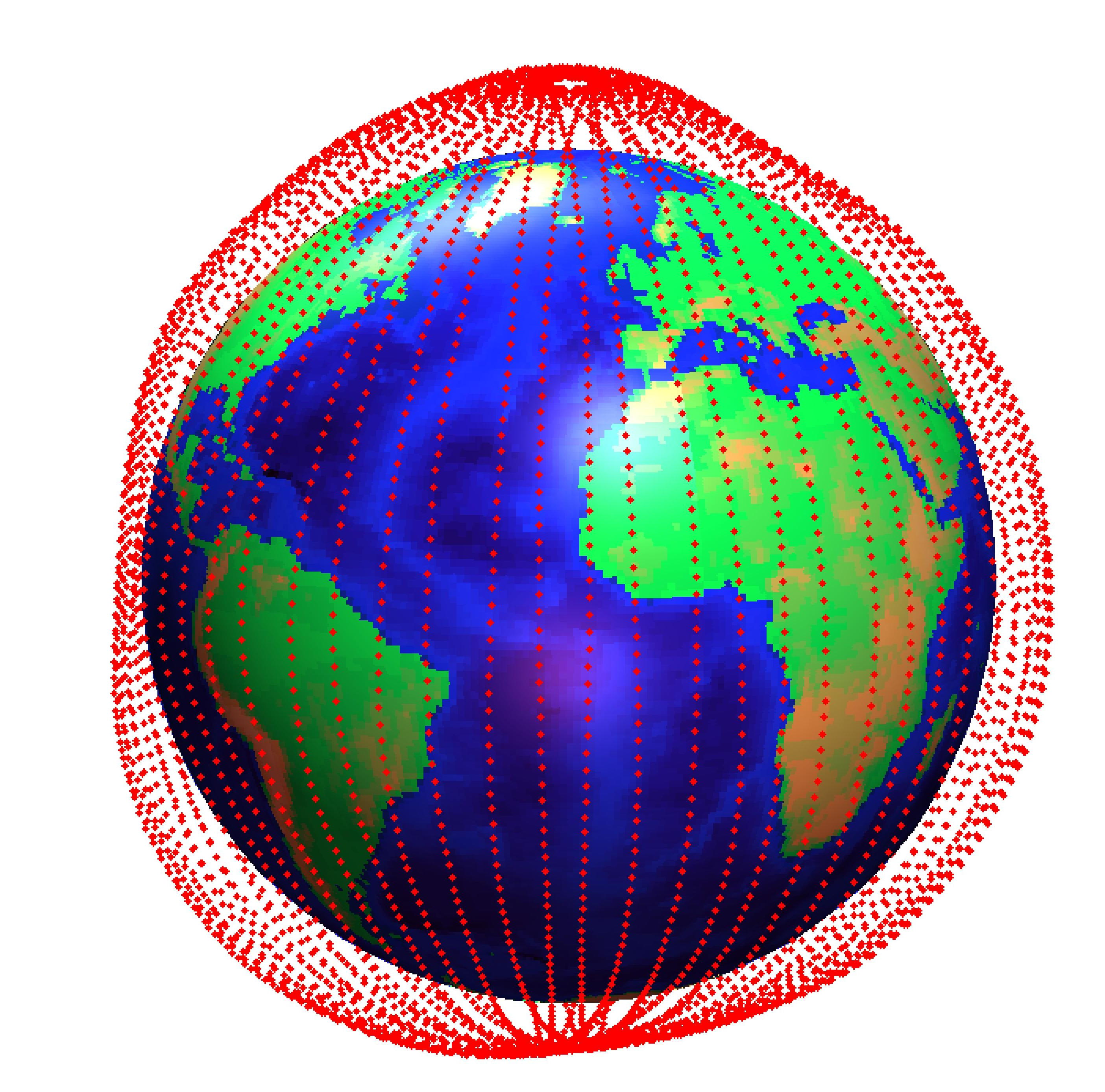}&
\includegraphics[width=0.43\textwidth]{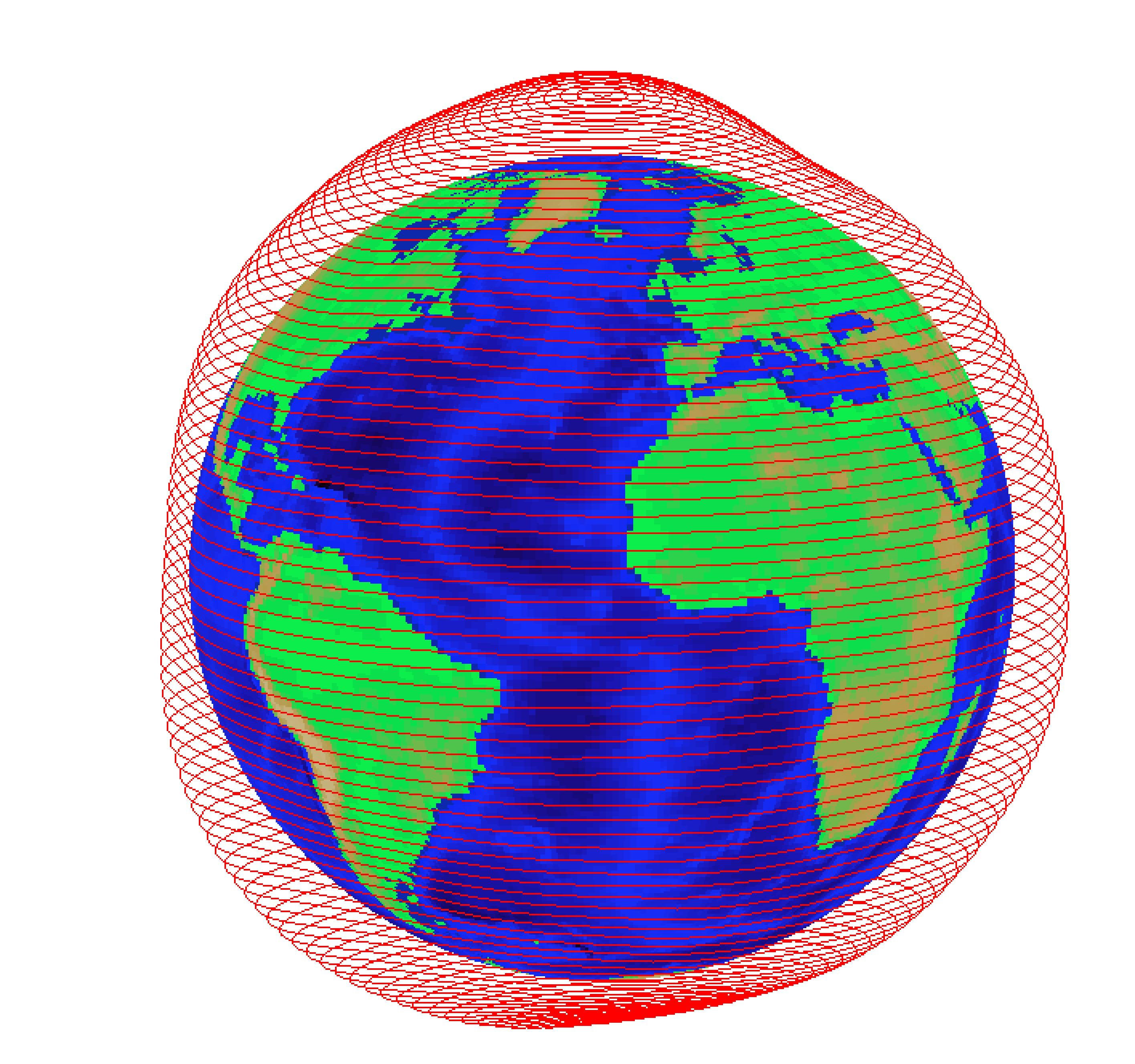}\\
\end{tabular}
\caption{Geopotential data values (left) and spherical spline surface of the data (right) }
\label{geofit} 
\end{figure}

One can see that spherical spline  (the right graph)  fits the given data set (the left graph) reasonably well.   
\end{example}

\subsection{Approximation Order of Spline Spaces}
Throughout the 1990s, Professor Schumaker extensively researched the approximation capabilities of both bivariate and spherical splines. A key outcome of this research was published in \cite{LS98}, a study that stands out in the field and is detailed below. 

\begin{theorem} [Lai and Schumaker 1998\cite{LS98}]
Suppose that $\triangle$ is a $\beta$-quasi-uniform 
triangulation of domain $\Omega\in {\bf R}^2$ 
and suppose that $d\ge 3r+2$.  Fix $0 \le m\le d$. 
Then for any $f$ in a Sobolev space $W^{m+1}_p({\Omega})$, there
exists a quasi-interpolatory spline $Q_f\in S^r_d(\triangle)$ such that 
$$
\|f- Q_f\|_{k,p,\Omega}\le C |\triangle|^{m+1-k}|f|_{d+1,p,\Omega}, 
\forall 0\le k\le m+1,
$$
for a constant $C>0$ independent of $f$, but dependent on $\beta$ and $d$.
\end{theorem}

This is a smooth version of the well-known Bramble-Hilbert lemma (cf. \cite{BH71}) which found the approximation order of spline space when the smoothness $r=0$.   

Another important work is the approximation order of spherical splines. See the following 
result.  
\begin{theorem}[Neamtu and Schumaker, 2004\cite{NS04}]
Let $d\ge 3r+2$ and $1\le p\le \infty$. Then there exists a constant $C$ dependent only 
on $d, p$ and the smallest angle in spherical triangulation $\triangle$ such that for any  
$f\in W^{m+1}_p(\mathbb{S})$
there exists a spline $Q(f)\in S^r_d(\triangle)$ with 
$$
|f- Q(f)|_{k,q,\mathbb{S}} \le C\,  |\triangle|^{d+1-k}~  |f|_{d+1,p, \mathbb{S}}
$$
for all $0 \le k \le d$ such that $Q(f) \in W^k_p(\mathbb{S})$.
\end{theorem}

There are many other studies on approximation properties of multivariate splines by Professor Schumaker 
and his collaborators. Notably, the approximation characteristics of discrete least squares in bivariate spline spaces with stable bases are examined in \cite{GS02}. Additionally, the study of minimal energy spline functions and their approximation properties is detailed in \cite{GLS02}. The application of domain decomposition methods in discrete least squares, minimal energy splines, and penalized least squares splines is also a topic of interest, as explored in \cite{LS09}. Furthermore, a comprehensive overview of fundamental concepts in bivariate, trivariate, and spherical splines is presented in the monograph \cite{LS07}, which is illustrated in Figure~\ref{book}.
\begin{figure}[htpb]
\centering
\includegraphics[width = 0.8\textwidth]{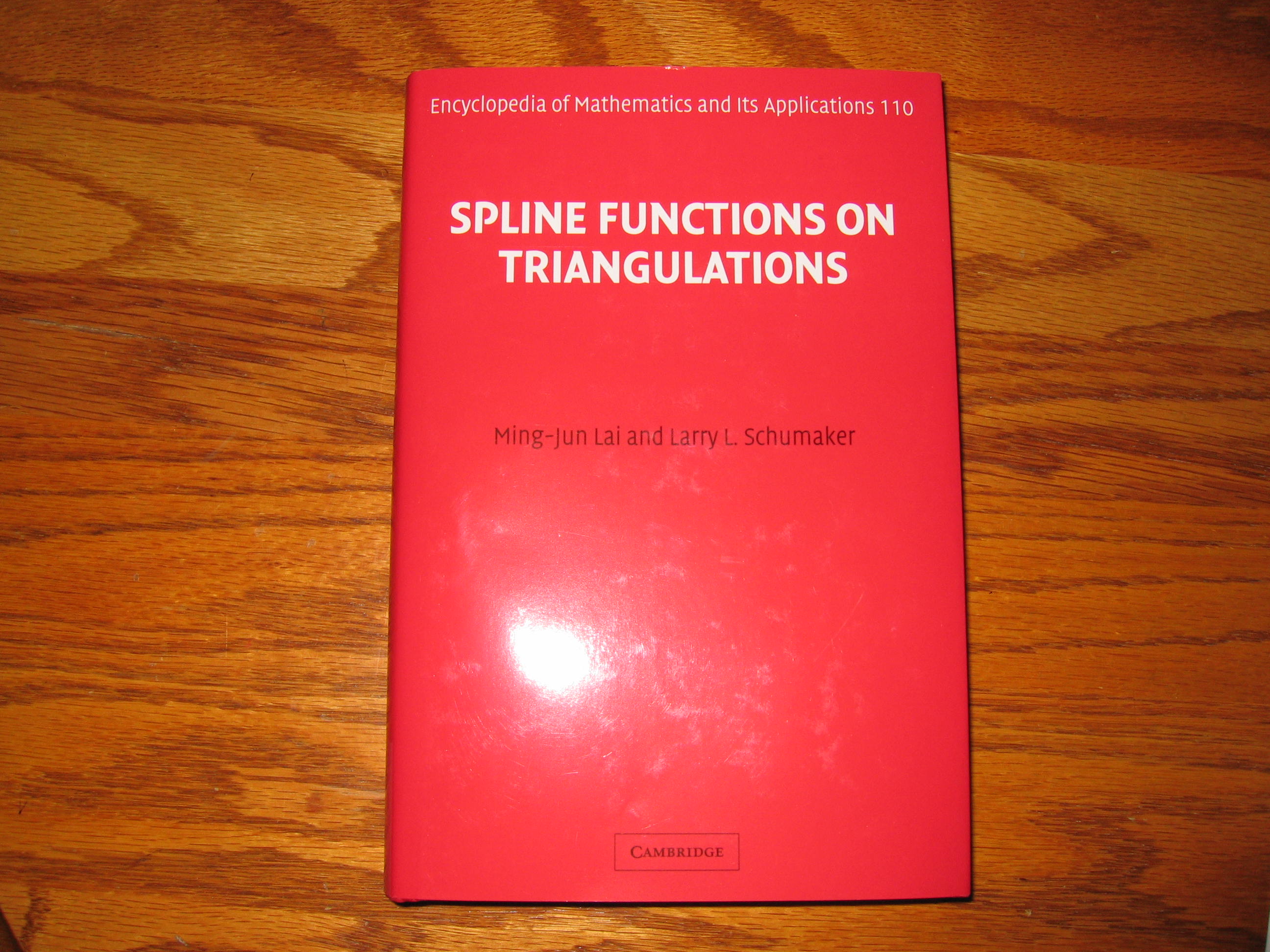}
\caption{A monograph on multivariate splines} 
\label{book}. 
\end{figure}

A persisting challenge in the field lies in determining the approximation order of trivariate spline spaces within the $L^p$ norm range of $1 \leq p < \infty$. While initial investigations into the approximation order in the $L^\infty$ norm were conducted, as outlined in \cite{L89}, these studies primarily focused on cases where the degree $d$ is at least $6r+3$. However, there is limited understanding regarding the approximation order for degrees $d < 6r+3$ over special tetrahedral partitions, as far as the author is aware.

\section{Computational Methods for Multivariate Splines}
One of computational methods for multivariate splines  
is  the classic method of using basis functions like finite elements
 to form a spline subspace and construct a linear approximation of the given data or 
the solution of a given partial differential equation. 
 The various applications and nuances of this method have been compiled in the recent monograph \cite{S15} by Professor Schumaker. However, this approach encounters a significant challenge: the creation of smooth spline basis functions with $C^2$ or higher continuity (e.g., $C^r$ for $r \geq 3$) proves to be exceedingly complex. In fact, constructing $C^r$ spline basis functions typically necessitates using a spline degree of $d \geq 3r + 2$ over a general triangulation in $\mathbb{R}^2$. For $C^2$ spline basis functions, a minimum degree of $d = 8$ is required for $r = 2$. Such high degrees introduce a large number of coefficients (degrees of freedom) to determine, presenting significant challenges in their implementation and the resolution of these unknown degrees of freedom.

An alternative computational method, discussed in \cite{ALW06}, proposes using the coefficient vector of each discontinuous spline function of degree $d$ over a triangulation $\triangle$. This method integrates smoothness conditions as constraints in a minimization strategy, employed for tasks like data fitting, data interpolation, or the numerical solution of PDEs. Here, the minimization process is utilized to ascertain the unknown degrees of freedom specific to each problem. The procedural steps for this method are outlined below.

\begin{itemize}
\item One starts with a discontinuous spline space $S^{-1}_d(\triangle)$ with a triangulation $\triangle$ of the domain $\Omega$ of interest.  

\item Let ${\bf c}= (c_{T_1}, c_{T_2}, \cdots, c_{T_n})^\top$ be the 
representation of a spline function in  $S^{-1}_d(\triangle)$, where $c_{T_i}$ is  the 
coefficients of polynomial in B-form (cf. \cite{LS07}) 
over triangle $T_i\in \triangle, i=1, \cdots, n$.

\item 
Since the smooth conditions across each interior edge of $\triangle$ are linear conditions in terms of ${\bf c}$, 
one puts all the linear smoothness conditions over all interior edges together to form a matrix 
$H$. So $H{\bf c}=0$ if and only if the spline 
function with coefficient vector ${\bf c}$ is in $S^r_d(\triangle)$. 

\item Write $I {\bf c}={\bf f}$ to be the interpolating conditions  and/or $B{\bf c}= 
{\bf g}$ to   the boundary conditions if solving a boundary value problem of PDE.  
One adds all the constraints in, i.e.  $I {\bf c}={\bf f}$,  $H{\bf c}=0$, 
and $B{\bf c}= 
{\bf g}$ and then  solves a 
constrained minimization problem. 

\item The constrained minimization problem can be solved by using an iterative algorithm 
described in \cite{ALW06}.
\end{itemize} 

Note that the iterative algorithm only needs a few, say 3 or less iterations.  
 This computational method allows one to use smooth 
spline functions easily. Therefore, multivariate splines of arbitrary degree, 
arbitrary smoothness over arbitrary triangulation or tetrahedralization or spherical 
triangulation can be used for any applications. Some of applications will be 
discussed in the remaining part of this paper.

\section{Solutions of Linear Partial Differential Equations}
For convenience, how to solve the elliptic equations using multivariate splines will be
explained. 
Consider the Poisson equation over a bounded domain $\Omega$ 
in $\mathbb{R}^d$ for $d=2$ or $d=3$:
\begin{align}
-\Delta u &= f, \quad ~in~ \Omega \subset \mathbb{R}^d\\
u&=g,\quad ~on~ \partial \Omega,
\end{align}
where $\Delta=\frac{\partial^2}{\partial x^2}+\frac{\partial^2}{\partial y^2}$ or 
$\Delta=\frac{\partial^2}{\partial x^2}+\frac{\partial^2}{\partial 
y^2}+\frac{\partial^2}{\partial x^2}$. 

In general,  consider  second order elliptic PDE in non-divergence form:
\begin{equation}
\label{GPDE2}
\left\{
\begin{array}{cl} \sum_{i,j=1}^d a^{ij}(x)\frac{\partial}{\partial x_i}\frac{\partial}{\partial 
x_j}u+\sum_{i=1}^d b^{i}(x) \frac{\partial}{\partial x_i}u+c(x)u&= f, \quad x \in 
\Omega, \cr 
u&=g,  \quad \hbox{ on } \partial \Omega,
\end{array}\right. 
\end{equation}
where the coefficient functions $a^{ij}(x),b^i(x),c(x), i, j=1, \cdots, d$ 
are in $L^\infty(\Omega)$ and satisfy the elliptic condition.
A multivariate spline based collocation method is introduced in \cite{LL22}.  
For a given triangulation $\triangle$, one chooses a set of domain points 
$\{ \xi_i\}_{i=1,\cdots, N}$ over $\triangle$ as collocation points 
and finds the coefficient vector $\textbf{c}$ of spline function $\displaystyle s =\sum_{t\in 
\triangle}\sum_{|\alpha|=D}c^t_{\alpha} \mathcal{B}^t_\alpha$
 satisfying the following equation at those points
\begin{equation}
\label{PDE2}
\left\{
\begin{array}{cl}-\Delta  s(\xi_i) &= f(\xi_i), \quad \xi_i \in \Omega\subset \mathbb{R}^2 \cr 
s(\xi_i) &=g(\xi_i),  \quad \hbox{ on } \partial \Omega,
\end{array}
\right.
\end{equation}
where $\{ \xi_i=(x_i,y_i) \}_{i=1,\cdots, N} \in \mathcal{D}_{D',\triangle}$  are the domain 
points of $\triangle$ of degree $D'$.  Note that $D'$ may not be equal to $D$. 

Using these points, one has the following matrix equation:
$$
-K\textbf{c}:=\begin{bmatrix}-\Delta \mathcal{B}^t_\alpha (x_i,y_i) \end{bmatrix}
\textbf{c}=[f(x_i, y_i)]=\textbf{f},
$$
where $\textbf{c}$ is the vector consisting of all spline coefficients $c^t_\alpha, 
|\alpha|=D, t\in \triangle$. For a general second order elliptic PDE, 
the formula will be more complicated, but the ideas are the same.

In general, the spline $s$ with coefficients in $\bf{c}$ is a discontinuous function.  In order 
to make $s\in \mathcal{S}^r_D$, its coefficient vector $\textbf{c}$ must satisfy 
the constraints $H\textbf{c}=0$ for the smoothness conditions that the $\mathcal{S}$ 
functions  possess. 
Our spline collocation method is to find the minimizer ${\bf c}^*$ by solving the following 
constrained minimization: 
\begin{align}
\label{min1}
\min_{\bf c} J(c)&=\frac{1}{2}(\|-K{\bf c}- {\bf f}\|^2)\\ &\text{subject to } B \textbf{c} = 
\textbf{G}, H \textbf{c} = \textbf{0},
\end{align}
where $B, {\bf G}$ are from the boundary condition and $H$ is from the smoothness condition. $K$ 
may not be invertible. 

Based on spline approximation theorem, one can show that a neighborhood of  $-K \textbf{c}= 
{\bf f}$, i.e. 
\begin{equation}\label{neighbor}
\mathbb{N}_\epsilon = \{\textbf{c}: ||-K \textbf{c}-\textbf{f}||\le 
\epsilon, ||H\textbf{c}||\le \epsilon ,||B\textbf{c}- \mathbf{G}||\le \epsilon \}
\end{equation}
is not empty. Therefore, the minimization problem will have a  solution. Since the minimizing 
functional is strictly convex, the solution is unique.  

To show the multivariate spline collocation method works for  the 3D Poisson equation., 
 consider the  following domains of interest in Figure~\ref{domains}.  In Table~\ref{Times},
 the computational times for generating all necessary matrices, i.e. $K, B, H$ are given. 
\begin{figure}
  \includegraphics[width=.30\linewidth]{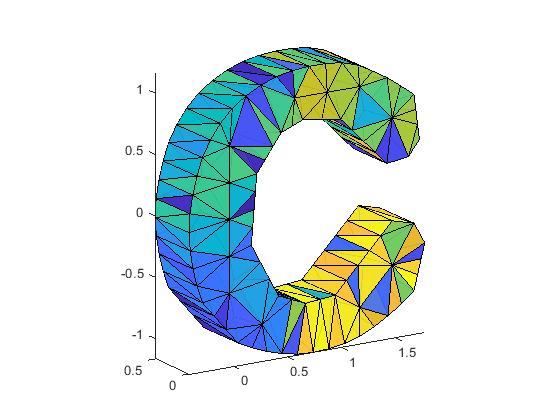}
  \includegraphics[width=.30\linewidth]{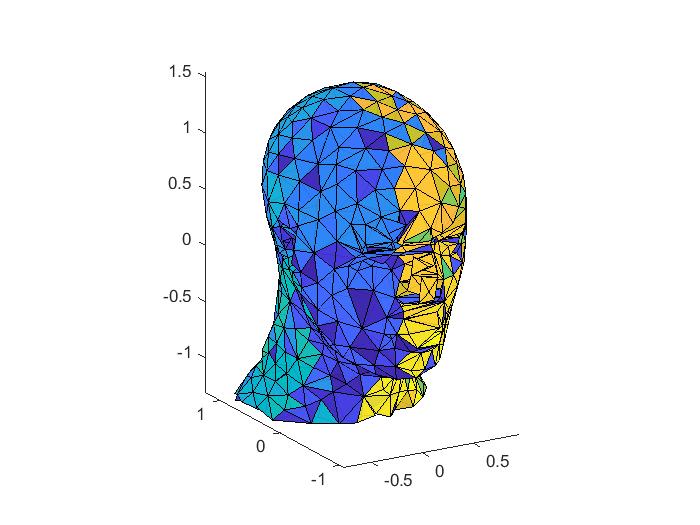}
  \includegraphics[width=.30\linewidth]{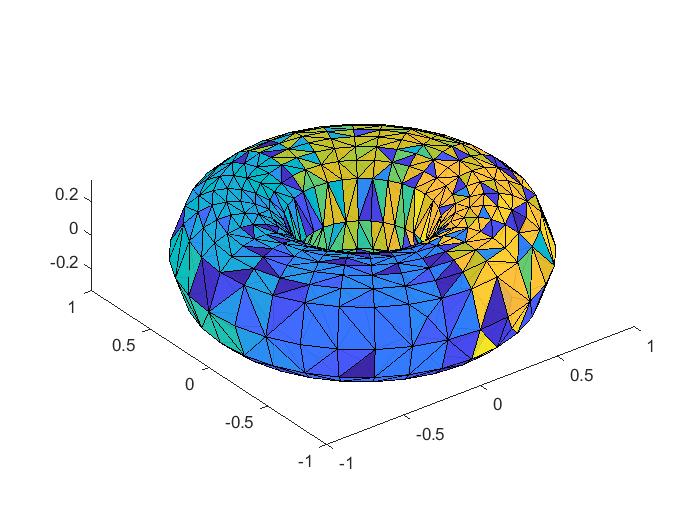}
  \caption{Three computational domains in $\mathbb{R}^3$ }\label{domains}
\end{figure}

\begin{table}
\centering
\begin{tabular}{ c |c c c c  } 
\hline
 Domain& Number of  & Number of  &CPU time \\
      &vertices& tetrahedron  &(seconds)\\
\hline
Letter C& 190& 431 & 11.70  \cr   
  Human head& 913&1588 & 44.9  \cr  
 Torus& 773&2911 &  442   \cr  
\hline
\end{tabular}
\caption{Times in seconds for generating necessary matrices using 48 processors with $D=9$ for each 3D domain in Figure
\label{domains}}
\label{Times}
\end{table}


The following 10 testing smooth and non-smooth solutions are 
used to test the spline solution of  the 3D Poisson equation:
\begin{eqnarray*} 
u^{3ds1}&=&\sin(2x+2y)\tanh(\frac{xz}{2})\cr 
u^{3ds2}&=& e^{\frac{x^2+y^2+z^2}{2}}\cr  
u^{3ds3}&=& \cos(xyz)+\cos(\pi(x^2+y^2+z^2))\cr
u^{3ds4}&=&\frac{1}{1+x^2+y^2+z^2}\cr
u^{3ds5}&=& sin(\pi(x^2+y^2+z^2))+1\cr
u^{3ds6}&=& 10e^{-x^2-y^2-z^2}\cr 
u^{3ds7}&=& \sin(2\pi x)\sin(2\pi y)\sin(2\pi z)\cr
u^{3ds8}&=& z\tanh((-\sin(x)+y^2))\cr 
u^{3dns1}&=& |x^2+y^2+z^2|^{0.8}\cr 
u^{3dns2}&=& (xe^{1-|x|}-x)(ye^{1-|y|}-y)(ze^{1-|z|}-z).
\end{eqnarray*}

3D Numerical Results are given in Table~\ref{Poisson3D1}. 
    
\begin{table}[thpb]
\centering
\begin{tabular}{ c |c c |c c| c c } 
\hline
\multicolumn{1}{c|}{} &\multicolumn{2}{c|}{C shaped domain}&\multicolumn{2}{c|}{Human head}&
\multicolumn{2}{c}{Torus}  \\
\hline
Solution &RMSE& $\ell_\infty$ error&RMSE&  $\ell_\infty$ error&RMSE&  $\ell_\infty$ error\\ 
 \hline
$u^{3ds1}$&3.15e-11 & 9.69e-11 & 5.83e-12 &6.45e-11&1.79e-10 & 2.04e-09    \cr  
$u^{3ds2}$&8.21e-10 & 2.15e-09 & 3.45e-10 &2.95e-09&1.14e-08 & 8.50e-08   \cr  
$u^{3ds3}$&7.33e-10 & 2.37e-09 & 7.26e-10 &8.21e-09&5.34e-09 & 3.31e-08  \cr  
$u^{3ds4}$&3.89e-10 & 1.06e-09 & 2.68e-10 &2.76e-09&3.57e-09 & 2.29e-08   \cr  
$u^{3ds5}$&1.02e-09 & 2.88e-09 & 9.75e-10 &5.78e-09&1.33e-08 & 8.95e-08   \cr  
$u^{3ds6}$&3.86e-09 & 1.10e-08 & 2.35e-09 &2.47e-08&3.39e-08 & 1.90e-07   \cr  
$u^{3ds7}$&1.76e-09 & 1.49e-08 & 4.19e-08 &5.21e-07&1.01e-07 & 2.34e-06   \cr  
$u^{3ds8}$&5.89e-11 & 1.94e-10 & 2.69e-11 &1.66e-10&6.42e-10 & 4.32e-09    \cr  
$u^{3dns1}$&1.15e-06 & 9.60e-05 & 3.82e-06 &6.23e-04&5.07e-09 & 3.22e-08    \cr  
$u^{3dns2}$&5.49e-06 & 9.37e-05 & 2.30e-04 &4.84e-03&1.09e-04 & 1.58e-03   \cr  
\hline
\end{tabular}
\caption{The  root mean square error(RMSE) and maximum errors 
of spline solutions for the 3D Poisson equation 
over the three domains when  $r=1$ and $D=9$. These errors are computed based on $501^3$ equally-spaecd points which fall into the domains of interest.} \label{Poisson3D1}
\end{table}

This approach is also used for numerical solution of nonlinear PDE: Monge Amp\'ere Equation
in the 3D setting (cf. \cite{LL23}). Similar numerical results for general second order elliptic PDEs were obtained. 
Refer to \cite{LL22} for more details.

A convergence result of multivariate spline based collocation method is established 
in \cite{LL22}.  
\begin{theorem}[Lai and Lee, 2022\cite{LL22}]
	\label{mjlai05122021} 
Suppose that $(u-u_s)|_{\partial \Omega}=0$.  
Under the assumption that $\Omega$ has a uniformly positive reach,  we have the following inequalities:  
	\begin{eqnarray*}
		\|u-u_s\|_{L^2(\Omega)} \le C|\triangle|^2 \epsilon_1 \hbox{ and }
		\|\nabla (u-u_s)\|_{L^2(\Omega)} \le C|\triangle| \epsilon_1
	\end{eqnarray*}
	for a positive constant $C>0$, where $|\triangle|$ is the
	size of the underlying triangulation $\triangle$ and $\epsilon_1= \|\Delta u+ f\|_{L^2(\Omega)}.$ 
\end{theorem}

When $\Omega$ is a convex domain, $\Omega$ has a uniformly positive reach. Also, any star-shaped
domain has a uniformly positive reach. There are many non-convex domains, non-star-shapde 
domains which has a positive reach. We refer to \cite{GL20} for examples.

Advantages of the multivariate spline collocation method are 
\begin{itemize}
\item The spline collocation method in 2D/3D can be easily 
        implemented for various kinds of linear PDEs, e.g. biharmonic equations, Stokes equations, and etc.. 
\item One can easily use splines of high degree and enough smoothness as long as 
the computer memory allows.
\item One can choose collocation points to avoid the discontinuity from the PDE coefficients;
\item  One can increase the number of collocation points to enhance the solutions;
\item It does not need  weak formulation and weak solutions. 
\end{itemize}

Dr. J. Lee has experimented with the multivariate spline based collocation method extensively. 
See \cite{L23} for numerical results for biharmonic equation, Stokes equations, Keller-Segel 
system of partial differential equations, and etc.. 


\section{Construction of Smooth Curves and Surfaces}
Another interesting application of multivariate splines is to construct smooth curves 
and smooth surfaces. For convenience, one starts with a construction of curves.  
Suppose that one is given a data set (on the left) of Figure~\ref{Cex1}. 

\begin{figure}[htpb]
\centering
\includegraphics[width = 0.3\textwidth]{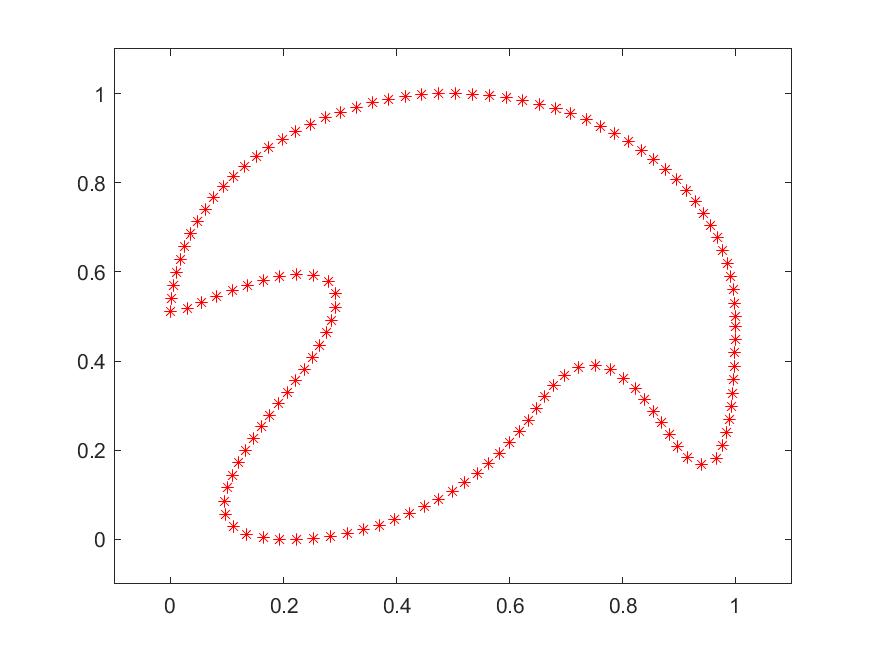}
\includegraphics[width = 0.3\textwidth]{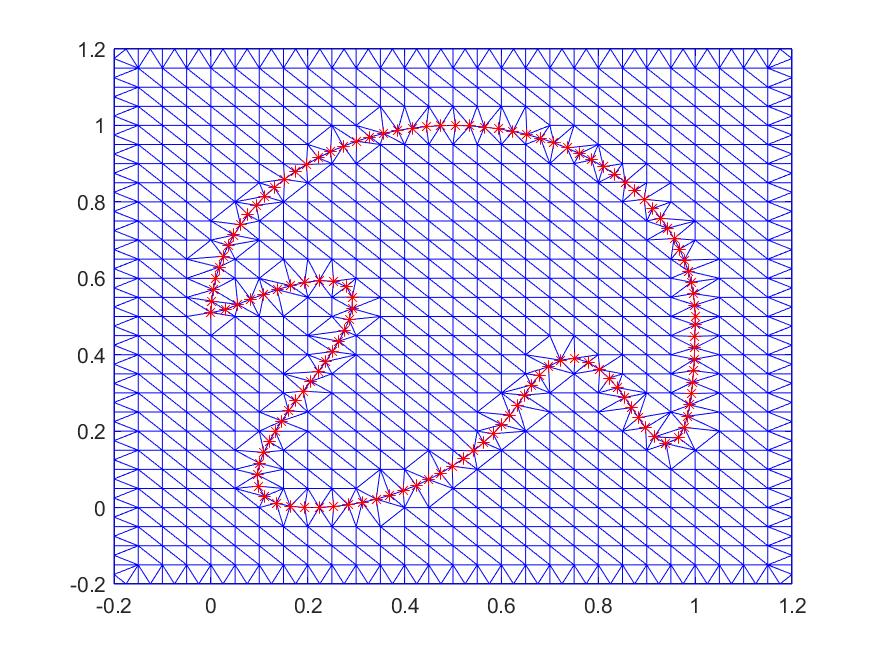}
\includegraphics[width = 0.3\textwidth]{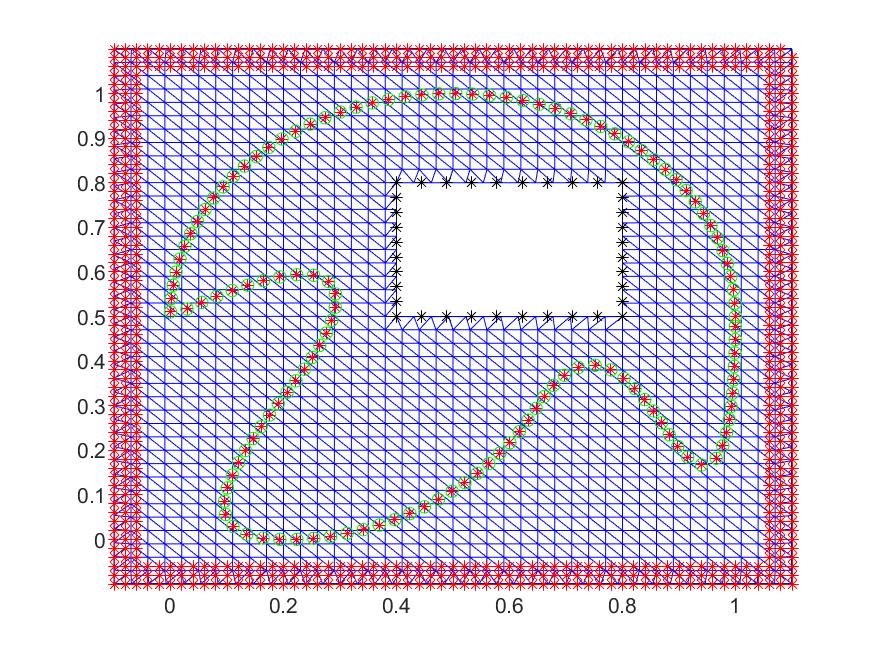}
\caption{A set of point locations (left) (courtesy Larry L. Schumaker),   
a constrained triangulation (middle) with piecewise linear interpolation as a part of edges 
and a triangulation with a hole (right) \label{Cex1}}
\end{figure}  

One approach is to construct an interpolatory spline solving the following 
minimization problem as discussed in \cite{DHLMX19}:
\begin{eqnarray}
\label{minenergy}
\min_{s\in S^r_d(\triangle)} \mathcal{E}(s), 
\hbox{s.t. } &&  s(x_i,y_i)=1, \forall (x_i,y_i)\in \mathcal{D}, \cr 
&&  s(x_j,y_j)=0, \forall (x_j,y_j)\in {\cal B},  \cr 
&&	s(x_k,y_k)=2, \forall (x_k,y_k)\in {\cal C}, 
\end{eqnarray}
where $\mathcal{E}(s)$ is the energy functional 
\begin{equation}
\label{E2}
\mathcal{E}(s) =\int_\Omega |\frac{\partial^2}{\partial x^2} s|^2 + 2
|\frac{\partial^2}{\partial x \partial y} s|^2 + |\frac{\partial^2}{\partial y^2} s|^2 ,
\end{equation}
$\mathcal{D}$ is a given point cloud, $\mathcal{B}$ is the boundary of the rectangular 
domain (in bold red), and 
$\mathcal{C}$ is the boundary of the hole in the middle of triangulation $\triangle$.  
One way to find this minimization is to solve the unconstrained minimization below.
\begin{eqnarray}
\label{minc}
\min_{s\in S^r_d(\triangle)} &&\sum_{(x_i,y_i)\in {\cal D}} 
|s(x_i,y_i)- 1|^2 + \sum_{(x_i,y_i)\in {\cal B}} |s(x_i,y_i)|^2 
	\cr
&& +\sum_{(x_i,y_i)\in {\cal C}}|s(x_i,y_i)-2|^2 + \lambda {\cal E}(s) 
\end{eqnarray}
with an appropriate parameter $\lambda>0$. This is so-called the penalized least squares 
method (cf. e.g. \cite{LS09}). The existence and uniqueness of the minimization 
(\ref{minc}) is well-known. 

\begin{example}[Multiple Curves]
In this example, there are multiple curves to describe a cartoon panda as shown in 
Figure~\ref{panda}. The method described 
in this section found all curves at once. 
\begin{figure}[htpb]
\centering
\includegraphics[width = 0.4\textwidth]{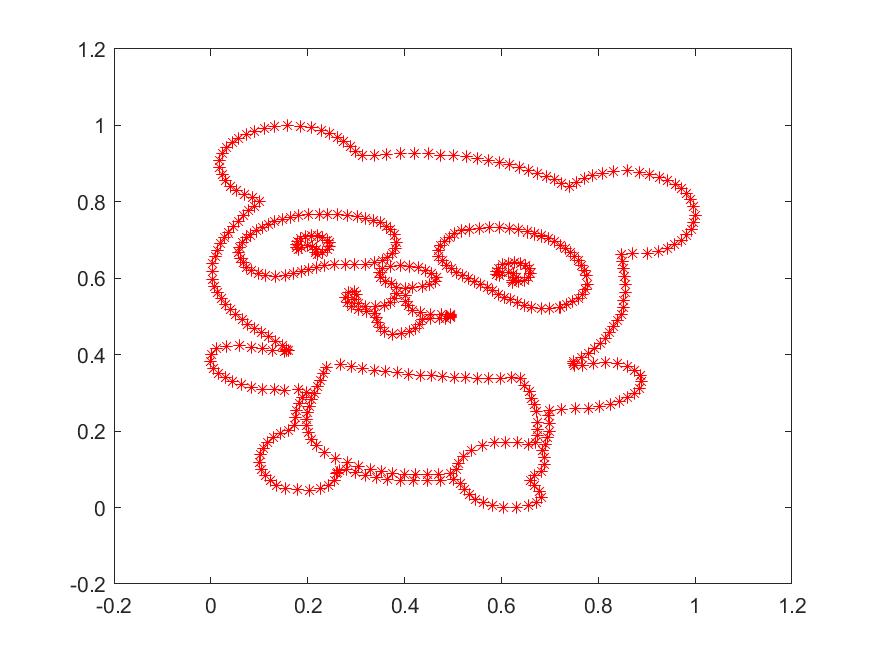}
\includegraphics[width = 0.4\textwidth]{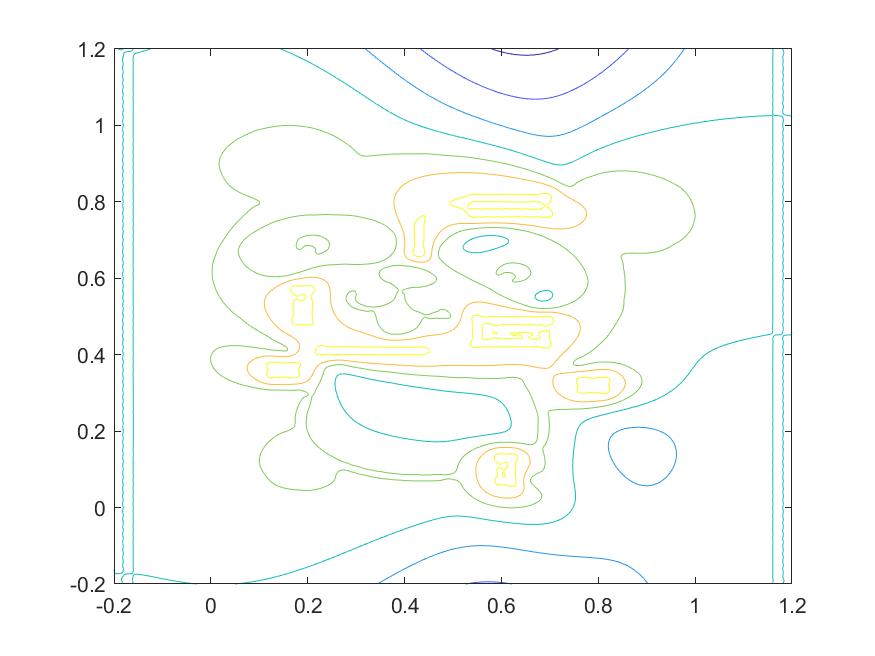}
\includegraphics[width = 0.4\textwidth]{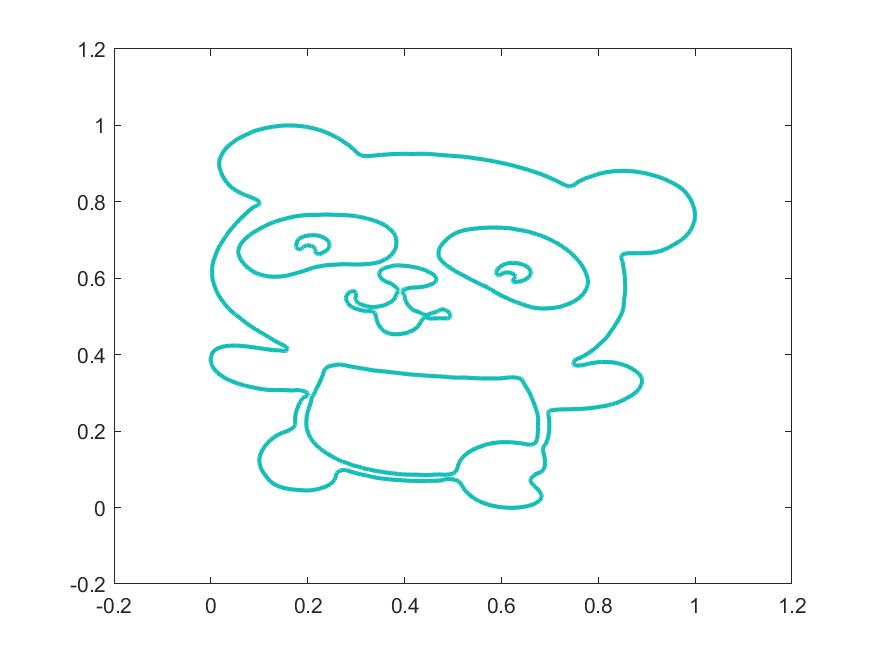}
\caption{A set of point locations (left),    
a contour plot (middle) of the spline minimizer of (\ref{minc})
and an interpolatory spline curve (right) \label{panda}}
\end{figure} 
\end{example}

\begin{example}[Curves with Corners]
Next consider a curve with multiple corners. One can choose a hole at the corner of the 
intended curve. Then the method in this section found a desired curve with corners.  See
\cite{X19} for detail.  
\begin{figure}[htpb]
\centering
\includegraphics[width = 0.4\textwidth]{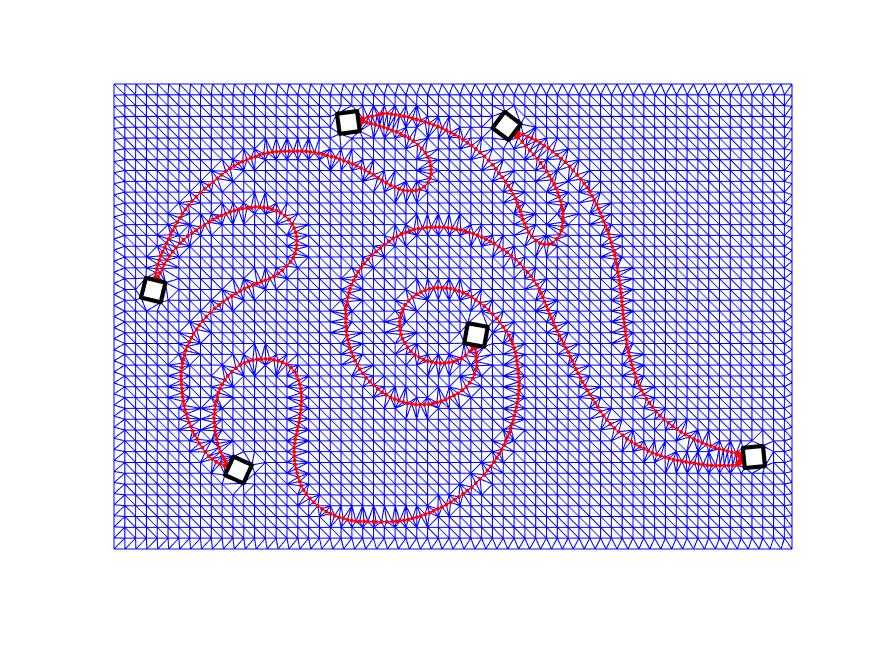} 
\includegraphics[width = 0.4\textwidth]{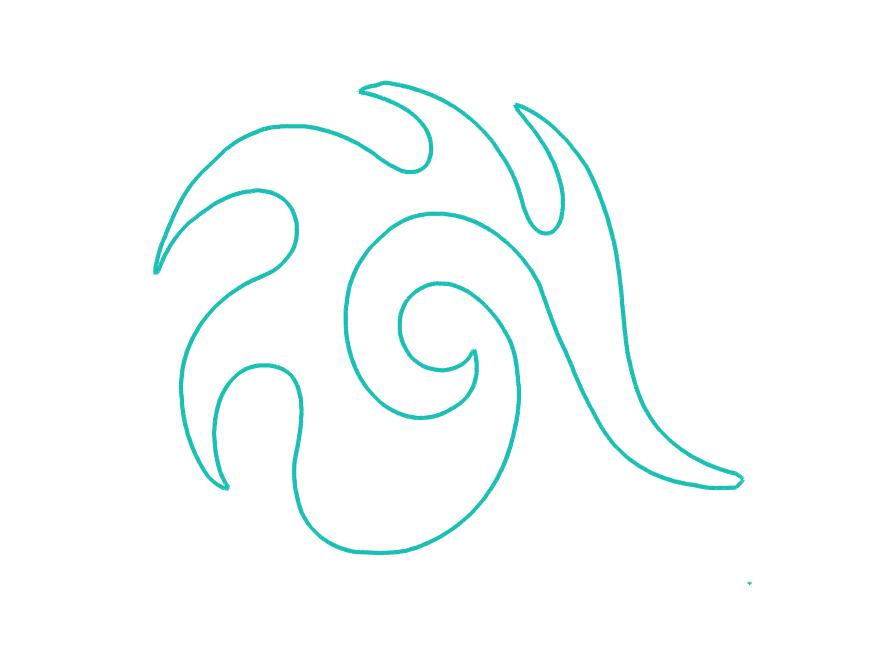} 
\caption{A constrained triangulation with holes at tips (left) and the level curve with corners. 
\label{ex7c}}
\end{figure}
\end{example}

\subsection{Trivariate Splines for Medical Data Fitting and Surface Construction}
A data set in $\mathbb{R}^3$ looks like a piece of blood vessel is given 
(courtesy of Bree Ettinger and Laura Sangalli). 

\begin{table}[htbp]
\begin{tabular}{c}
\includegraphics[width = 0.45\textwidth]{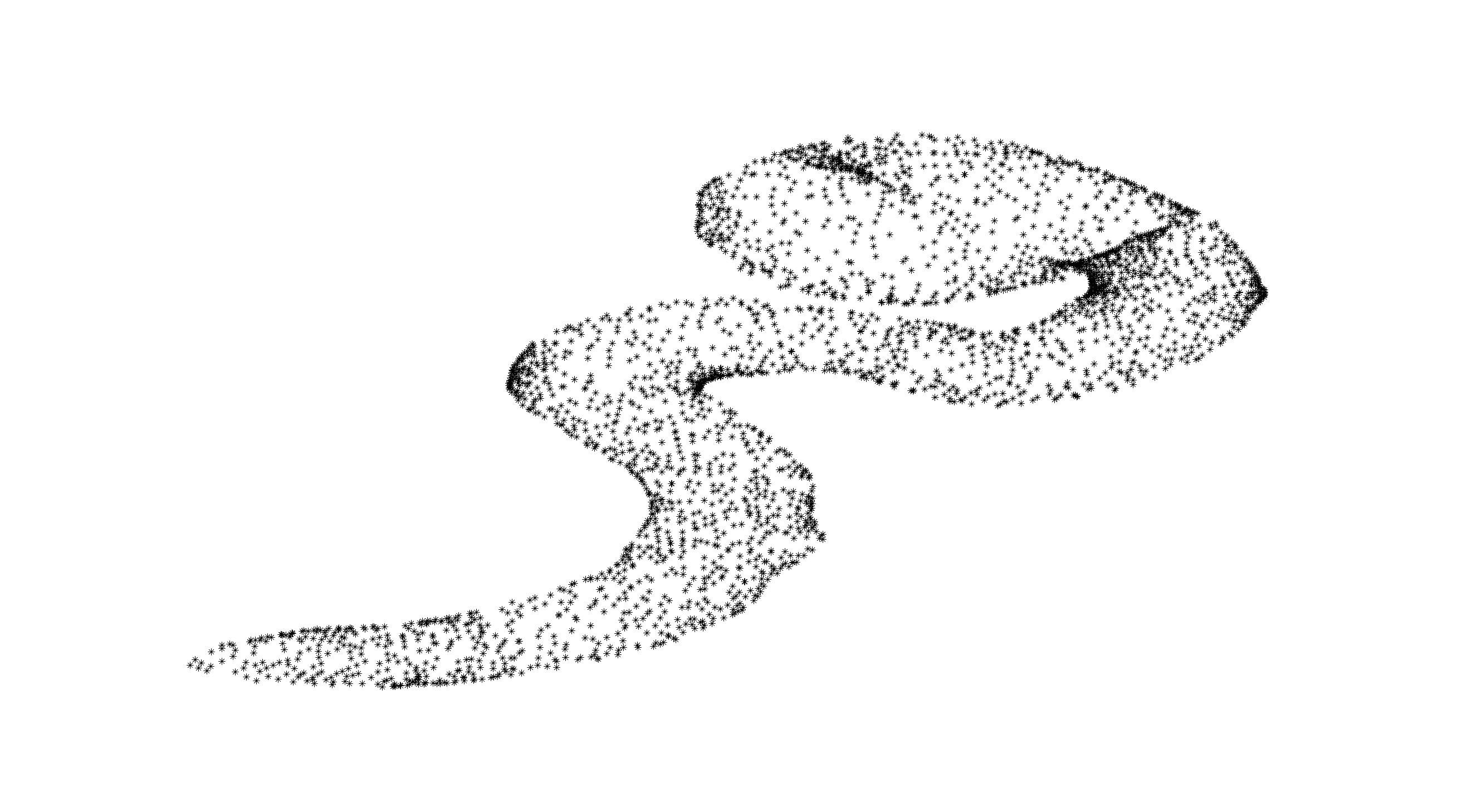} 
\includegraphics[width = 0.45\textwidth]{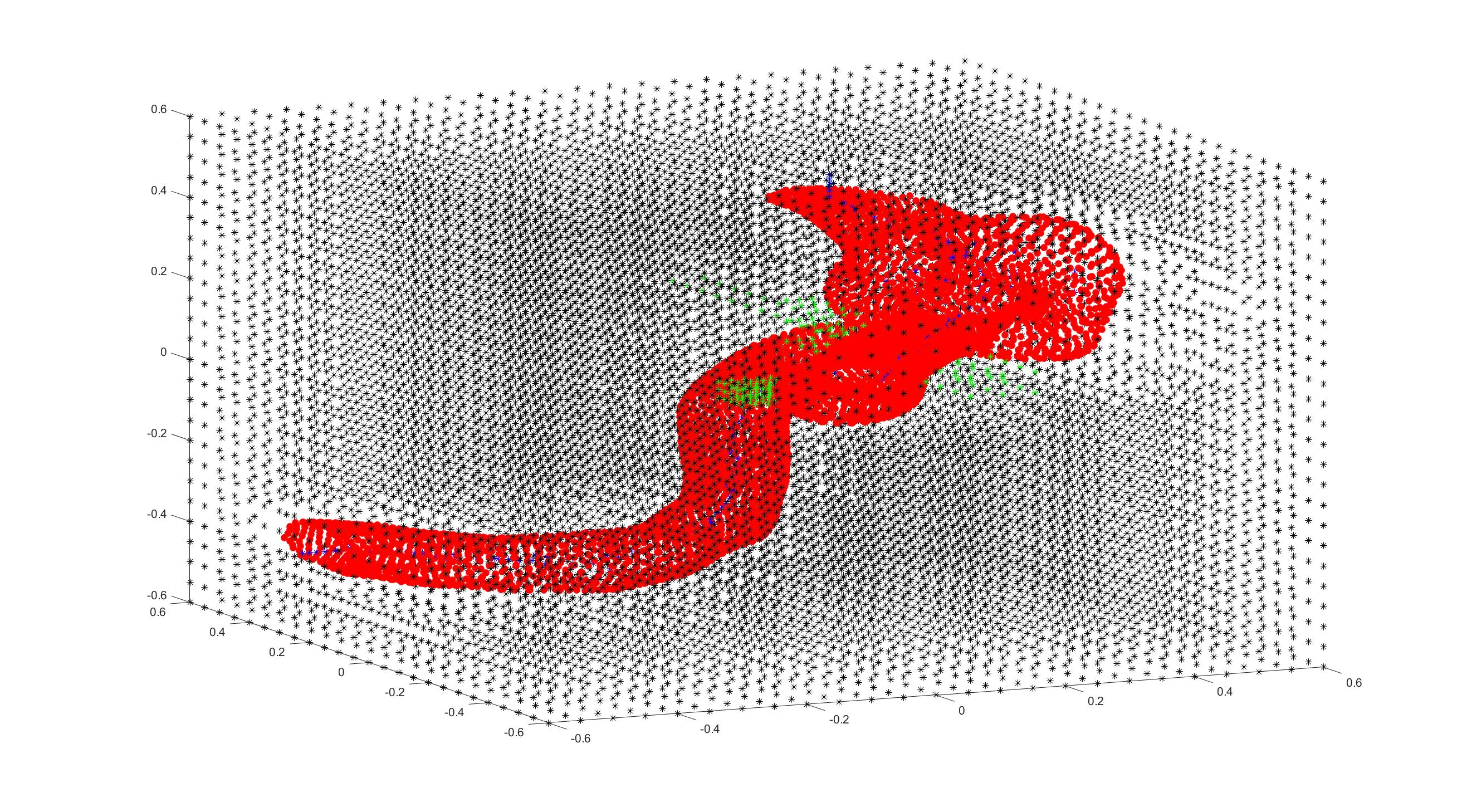}\cr
\end{tabular}
\caption{A data set in two different views}\label{vesseldata}
\end{table}

\begin{table}[htbp]
\begin{tabular}{c}
\includegraphics[width = 0.45\textwidth]{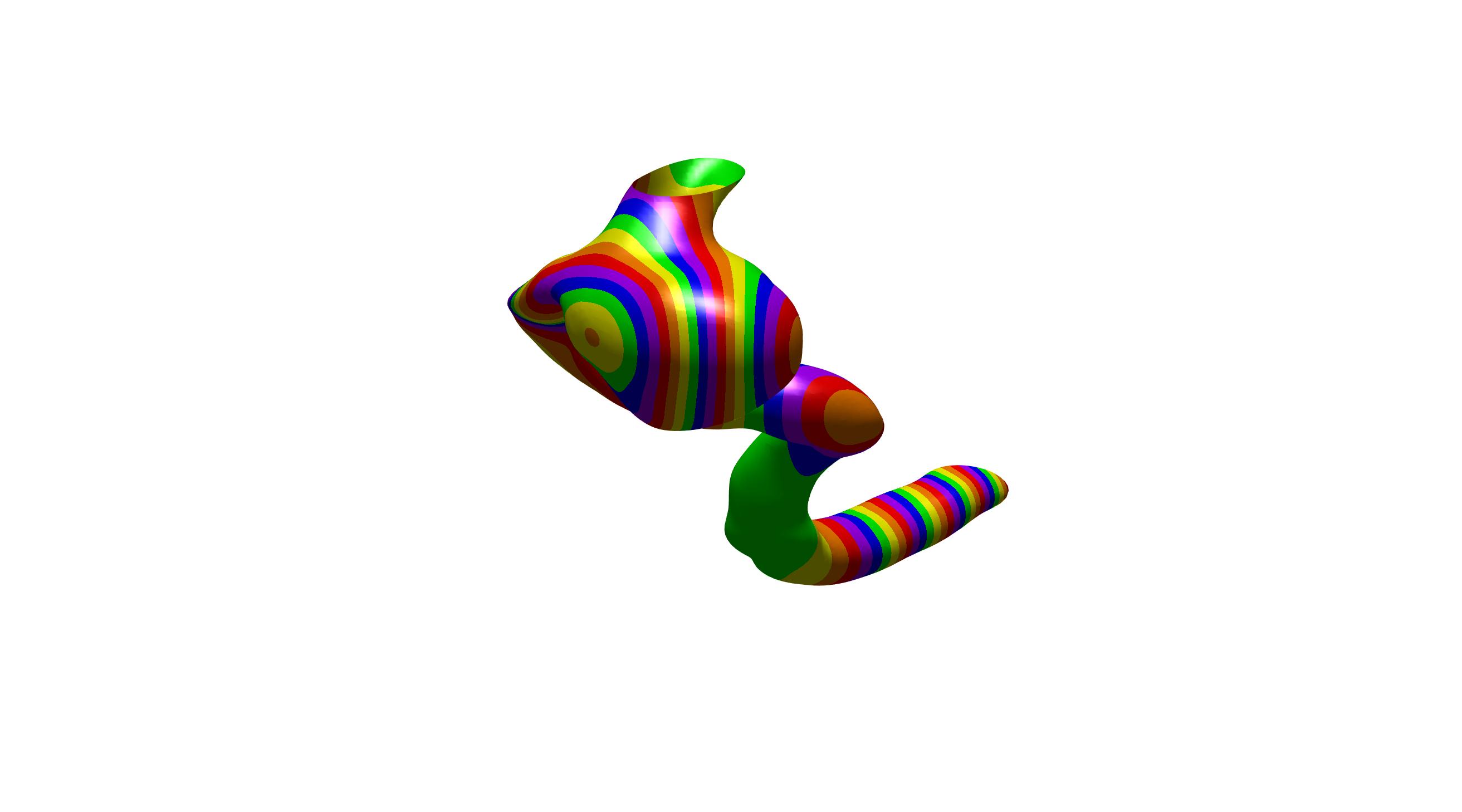} 
\includegraphics[width = 0.45\textwidth]{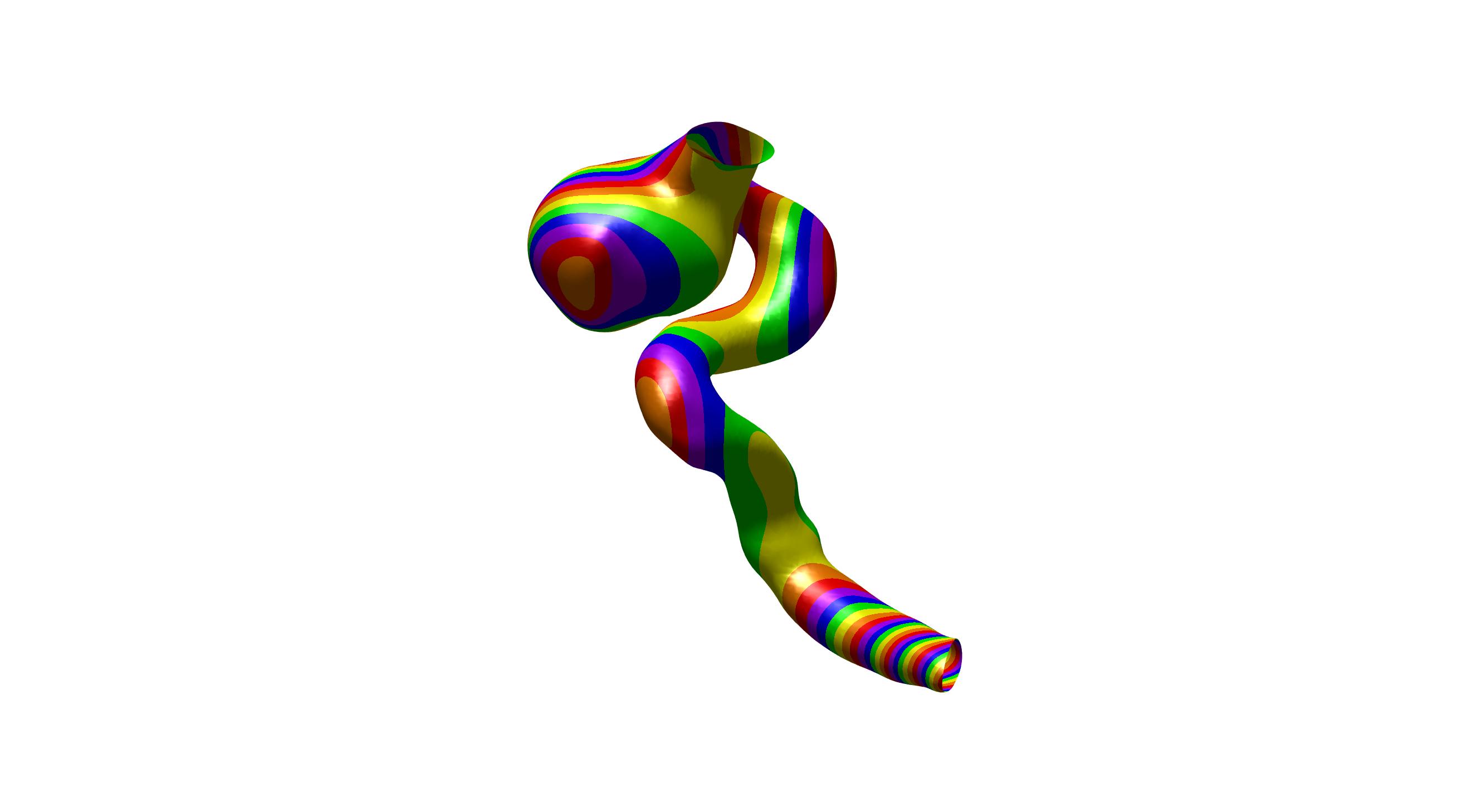}\cr
\end{tabular}
\caption{Two spline surfaces (two different views) of the blood vessel data
in Figure~\ref{vesseldata}}\label{vessel}
\end{table}
In \cite{DHLMX19}, the researchers use the penalized least squares method  to find the 
smooth surface interpolating the data in Figure~\ref{vesseldata} as shown in Figure~\ref{vessel}.  These surfaces are smooth and interpolate the given data nicely.

\begin{example}[Construction of Smooth Surfaces]
Certainly, one can create any data set and uses the penalized least squares method to 
find the desired surface. In Figure~\ref{ex3d}, four surfaces of various genius are shown. They 
are generated by using the method explained above. 
\begin{table}[htpb]
\begin{tabular}{cc}
\includegraphics[width = 0.4\textwidth]{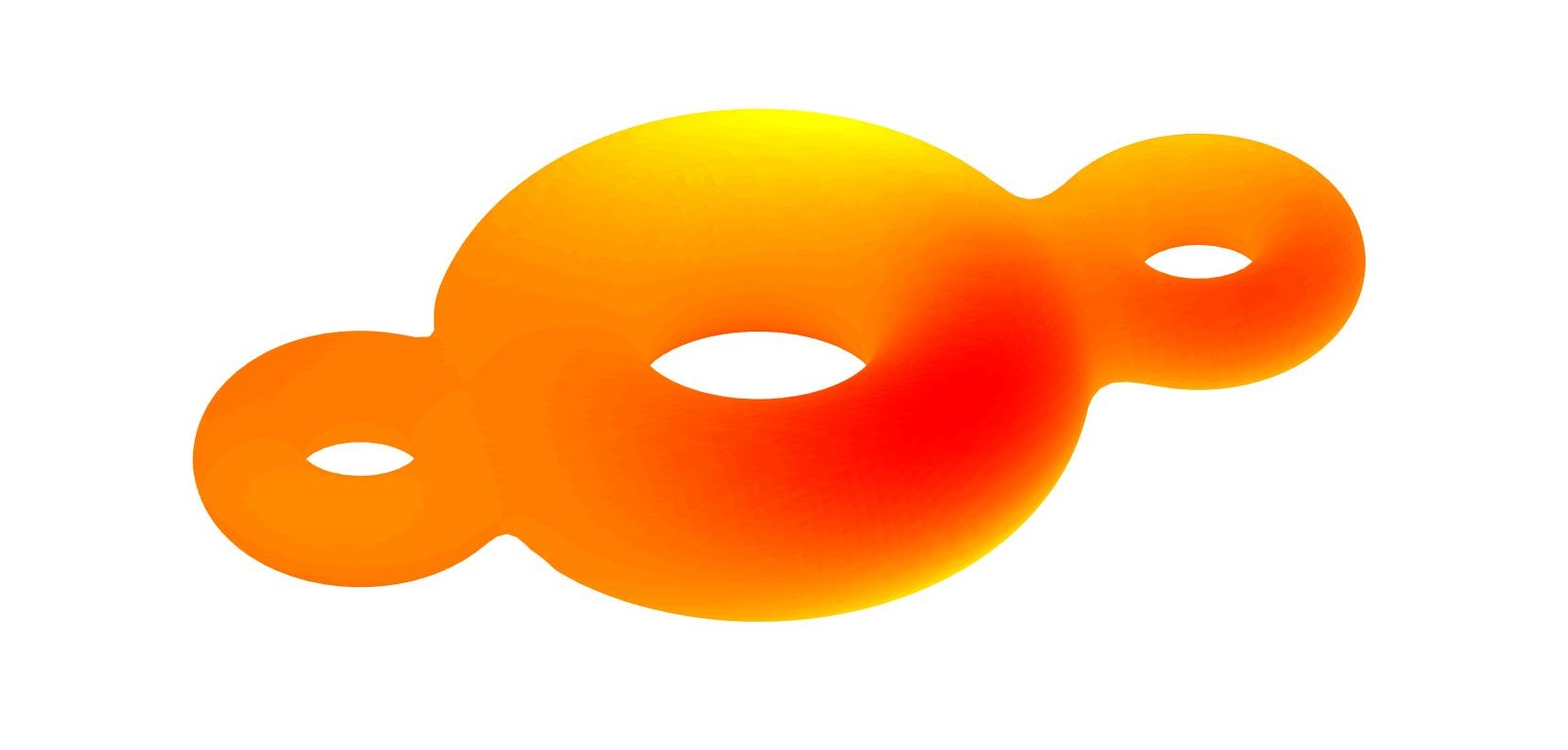} & 
\includegraphics[width = 0.4\textwidth]{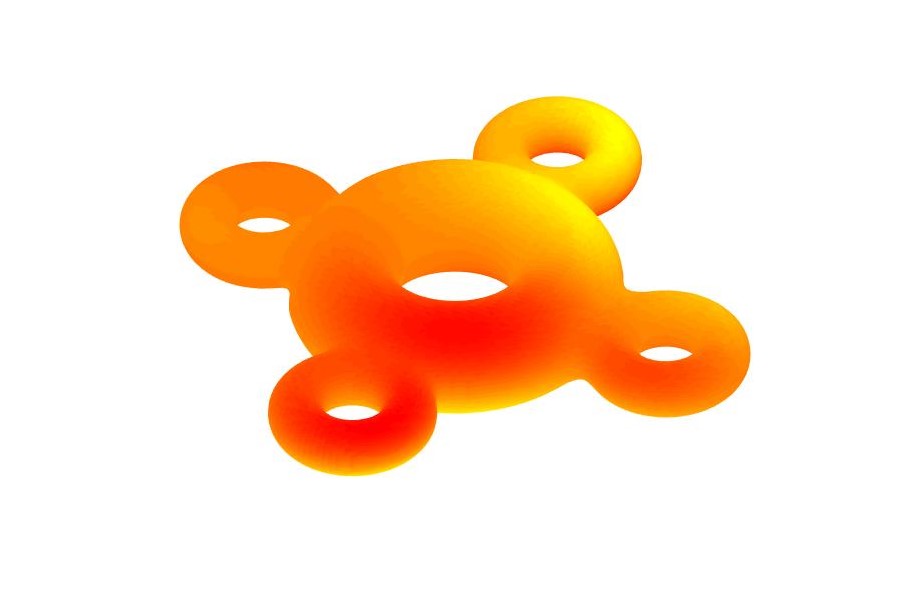} \cr 
\includegraphics[width = 0.4\textwidth]{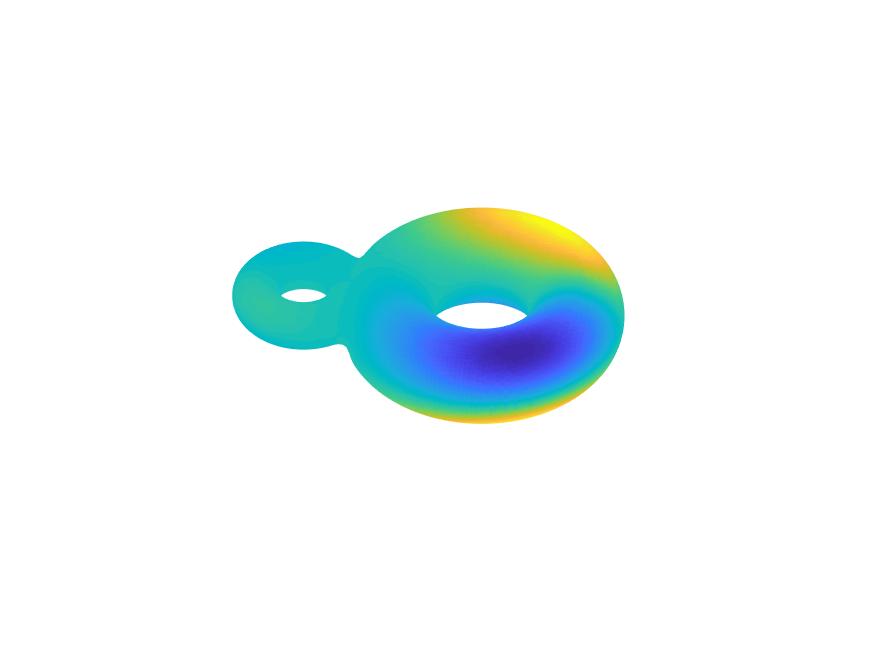} & 
\includegraphics[width = 0.4\textwidth]{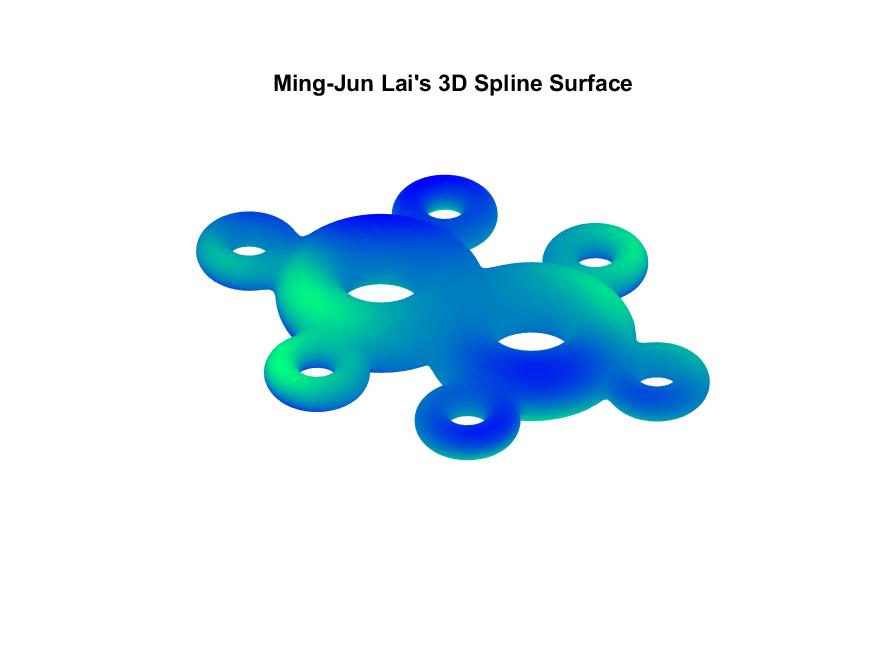}\cr
\end{tabular} 
\caption{Various genus 3  surfaces generated by 3D spline data fitting method} 
\label{ex3d}
\end{table}
The detail can be found \cite{DHLMX19}.  One can even print these surfaces out by using 3D printer. 
See Figure~\ref{ex3dA}.  
\begin{figure}[htpb]
\centering
\includegraphics[width = 0.22\textwidth]{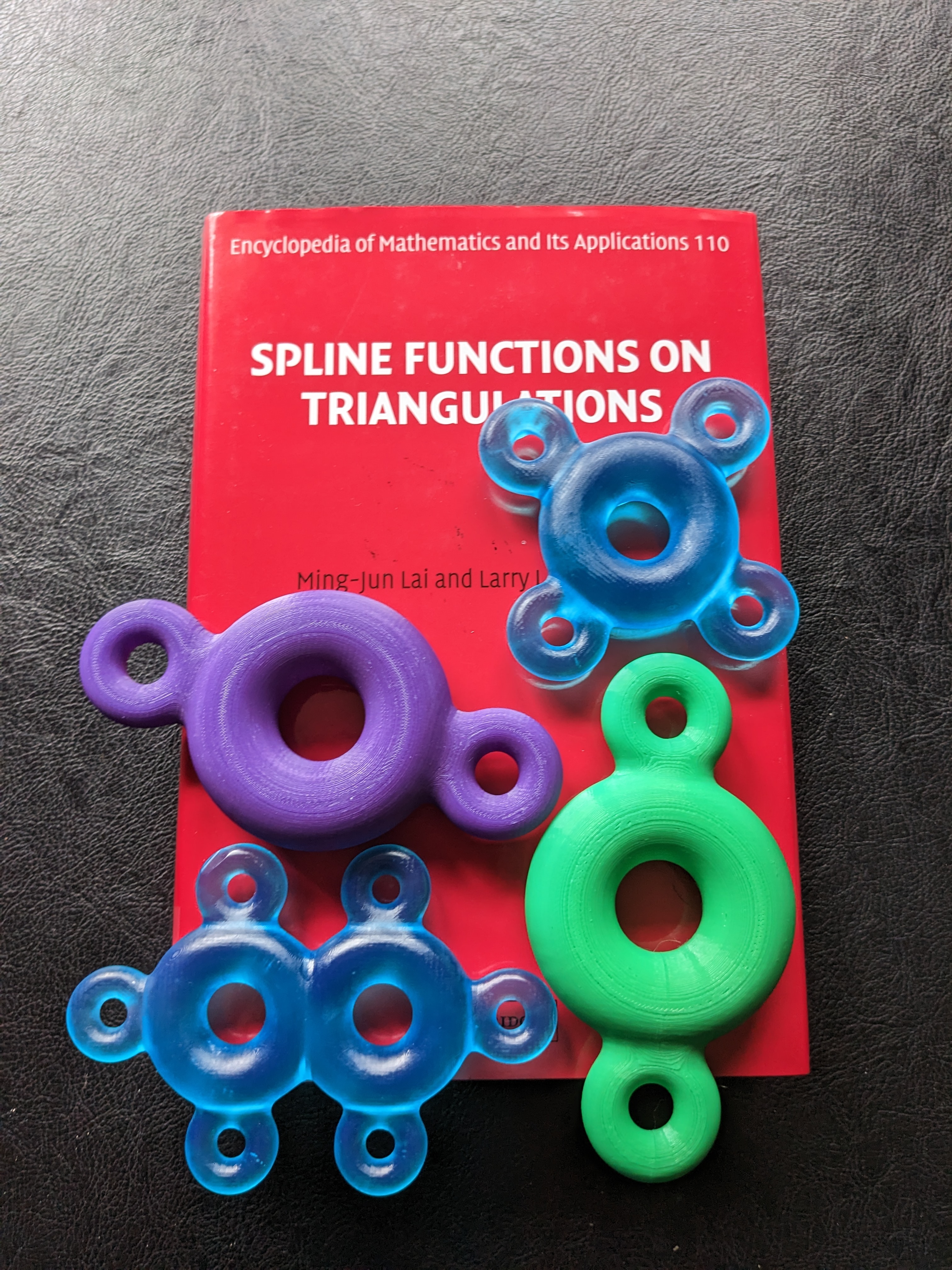} 
\caption{3 D models of the surfaces of various genus generated by 3D spline data fitting method }
\label{ex3dA}
\end{figure}
\end{example}




\section{Numerical Solution of the Monge-Ampere Equation}
The well-known Monge-Ampere Equation arises from the Monge's formulation of 
the optimal transportation problem. 
The well-known  optimal transportation problem can be described simply as follows. 
Assume that $f$ is a given density 
function over a domain $V\subset \mathbb{R}^d$ with $d \ge 1$.  
One has a pile $f$ of sands over a domain 
$V$  and a plan to move it to another location $W \subset \mathbb{R}^d$ with density function $g$.
Certainly, one assume that the volume $\int_V f(\bfx)d\bfx = \int_W g(\bfy) d\bfy$. 

For example, consider a file of sands forms an image (density) over a square domain $V$ 
which needs to be over a circular domain $W$ as shown in Figure~\ref{harper}.  
\begin{figure}
\centering
  \includegraphics[width=1\linewidth, height=5cm]{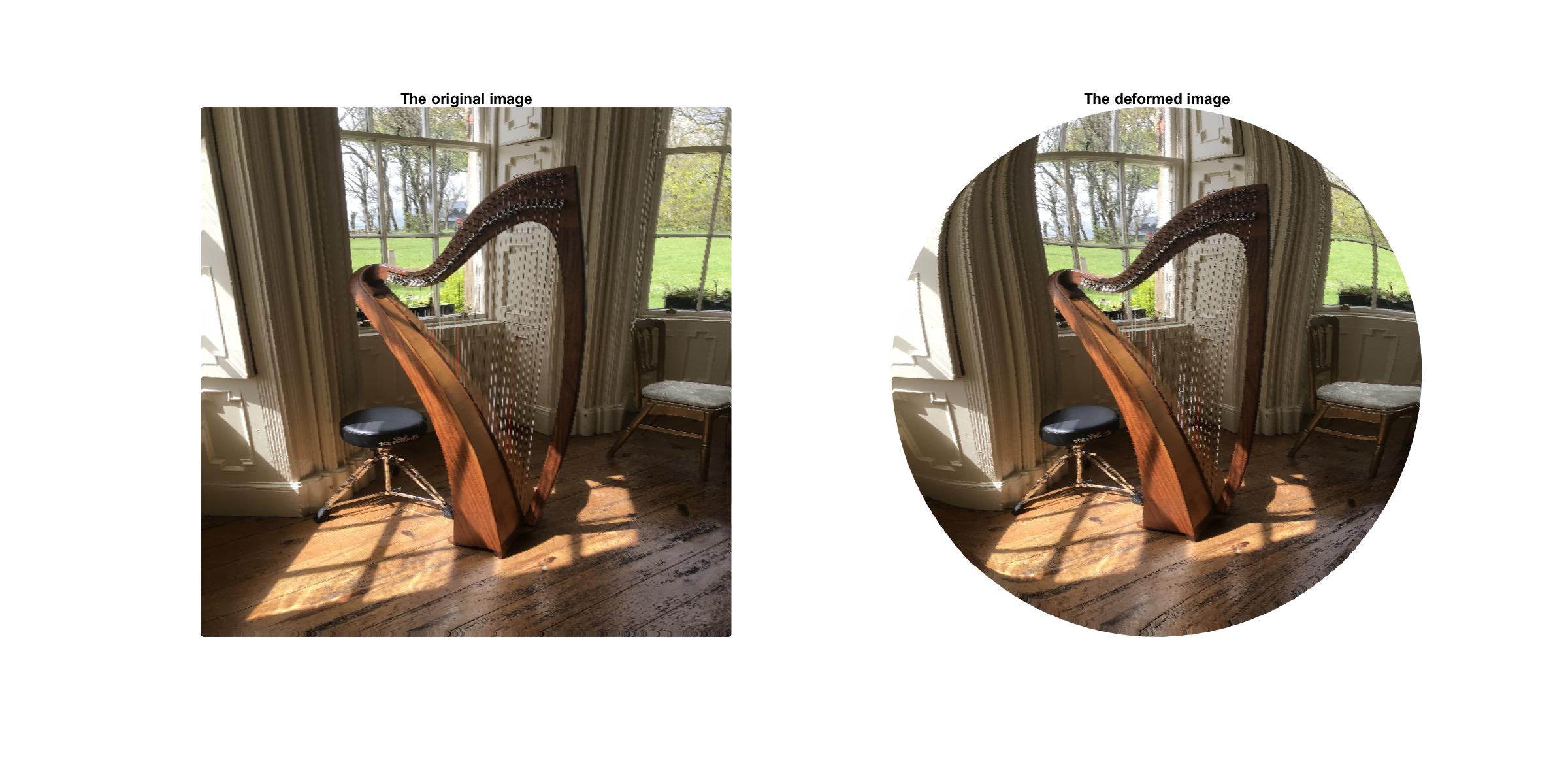}
  \caption{An image over a square domain is transported to a circular domain} \label{harper}
\end{figure}
The problem is to find the optimal way to transport the density image from $V$ (on the left) to 
$W$ (on the right). 

Clearly, moving the sand around needs some effort/energy which is modeled by a cost function 
$c(\bfx, \bfy)$ defined over $V \times W$. 
For convenience,  one uses 
\begin{equation}
\label{standard}
c(\bfx,\bfy)= \frac{1}{2} \|\bfx - \bfy\|^2.
\end{equation}

Let ${\cal T}(V,W)$ be the collection of all push-forwards which maps $f$ over $V$ to $g$ over
$W$. The optimal transport problem  can be recast as follows:
\begin{equation}
\label{Mproblem}
\min_{T\in  {\cal T}(V,W)}\int_V c(\bfx, T(\bfx)) f(\bfx) d\bfx 
\end{equation}
which is called Monge's optimal transportation problem. It was formulated in 1781 by Gasspard 
Monge and became famous since 1885 due to the prize offered by the Academy of Paris 
(cf. \cite{V03}). The problem has been studied for 250 years. One distinct research work is 
the Kantorovich  formulation of the optimal transportation problem based on 
\begin{equation}
\label{L1standard}
c(\bfx,\bfy)= \|\bfx - \bfy\|
\end{equation}
instead of (\ref{standard}). Dr. Leonid Kantorovich  was awarded a Nobel prize in 1978  
for his work in economics around 1945. In particular, his research initialized the study 
of linear programming. 

\subsection{Existence, Uniqueness, and Regularity}
There are many results known on the optimal transportation problem in (\ref{Mproblem}). The 
major result is summarized in the following   
\begin{theorem}[The Brenier Theorem]
\label{BrenierThm1}
Suppose that $f$ is a smooth function over $V$ or 
a positive density function which does not give a mass to small 
sets. There exists a unique push-forward 
$\nabla u$ with a convex function $u$ which is the minimizer of (\ref{Mproblem}) the cost function $c(\bfx,\bfy) = \frac{1}{2} \|\bfx - \bfy\|^2$.  Furthermore, $u$ satisfies a Monge-Amp\'ere equation:
\begin{equation}
\label{MAE}
\det(D^2u(\bfx)) = \frac{f(\bfx)}{g(\nabla u(\bfx))}, \forall \bfx\in V
\end{equation}
with a second boundary condition
\begin{equation}
\label{sbdc}
\nabla u(\bfx)= \bfy \in  W,~ \forall \bfx\in  V.
\end{equation}
\end{theorem}
\begin{proof}
Based on the existence theorem and uniqueness result in \cite{V03}, there exists an unique 
transportation $T$ and $T$ has the cyclical monotone property.  
According to the Rochafeller Theorem (cf. \cite{V03}), 
$T=\nabla \phi$ for a convex function such that for any measurable function $\eta$ on $W$, 
\begin{equation}
\label{key1}
\int_W \eta(\bfy) g(\bfy)d\bfy = \int_V \eta(T(\bfx)) f(\bfx) d\bfx = \int_V  
\eta(\nabla \phi (\bfx)) f(\bfx ) d \bfx
\end{equation} 
By using the change of variables on the left-hand side, we have
\begin{equation}
\label{key2}
\int_W
\eta(\bfy)g(\bfy)d\bfy =\int_W\eta(\nabla \phi(\bfx))g(\nabla u(\bfx)\det(D^2 u(\bfx)))d\bfx
\end{equation}
which should be the right-hand side of (\ref{key1}) for any measurable function $\eta$. It follows that
\begin{equation}
\label{key3}
f(\bfx) = g(\nabla u(\bfx) \det(D^2 u(\bfx)))
\end{equation}
which is the nonlinear Monge-Amp\'ere equation (\ref{MAE}). 
\end{proof}

There are many classic results on the regularity. The one of most recent works is 
\cite{CLW21}. 
\begin{theorem}
\label{Bound}
Assume that $V$ and $W$
are bounded convex domains in $\mathbb{R}^n$  with $C^{1,1}$ boundary and assume that 
$f \in C^\alpha(\bar{V})$ is positive with $0<\alpha<1$. Let $u$ be a convex solution to (\ref{MAE}) and 
(\ref{sbdc}). Then we have the estimate
\begin{equation}
\|u\|_{C^{2,\alpha}(\bar{V})} \le  C,
\end{equation}
where $C$ is a constant depending only $n, \alpha, f, V, W$. Furthermore, if $\alpha=0$, i.e. $f\in C^0(\bar{V})$, then  
\begin{equation}
\|u\|_{W^{2,p}(V)} \le  C
\end{equation}
for all $p\ge 1$, where $C$ is another constant dependent on  $n, p, f, V, W$.
\end{theorem}


\subsection{Numerical Solutions of the Monge-Amp\'ere equation}
We first note the difficulties of the Monge-Amp\'ere Equation. 
A solution to the (\ref{MAE}) is not easy. Note that 
$$
\det(D^2 u)=u_{xx} u_{yy} - (u_{xy})^2
$$
in $\mathbb{R}^2$ and
\begin{eqnarray*}
\det(D^2 u)= 
u_{xx} u_{yy} u_{zz} + 2u_{xy} u_{yz} u_{xz}-u_{xx}(u_{yz})^2- u_{yy} (u_{xz})^2-u_{zz} (u_{xy})^2
\end{eqnarray*}
in $\mathbb{R}^3$.  That is,  the equation is nonlinear. 

As the right-hand side of the equation (\ref{MAE}) is also dependent on 
$u$, the equation is fully nonlinear.

When $f$ and $g$ are  smooth, the solution $u$ is called the classic solution.  In this case, 
the solution $u\in C^2(\mathring{V})$ by using Theorem~\ref{Bound}.
When $f$ and $g$ are probability density functions and hence, $f$ and/or $g$ may not smooth,   
there are many versions of solution to (\ref{Mproblem}) with various cost functions 
$c(\bfx, \bfy)=\|\bfx- \bfy\|^p$ for $p\ge 1$:
\begin{itemize}
\item viscosity solution,
\item Aleksandrov solution,
\item Brenier solution,  
\item Pogorelov solution and
\item etc. 
\end{itemize}
Numerical solutions of the Monge-Ampere equation in the general setting is hard to find. More
study will be carried out.  Fortunately, there is a computational method based on finite difference 
discretization in \cite{BFO10} which is a very 
effective method for numerical solution of the Monge-Ampere equation (\ref{MAE}). The idea is to 
solve  the Poisson equation  iteratively. Based on this idea, bivariate and trivariate splines
were used in  \cite{LL23} and \cite{LL24} for solving the Poisson equation iteratively which produce 
many interesting numerical examples.
A few computational results in \cite{LL23} and \cite{LL24} are included. See Figure~\ref{3dMAE}, 
Figure~\ref{Bday} and Figure~\ref{Escher}.

\begin{figure}[h]
    \centering
         \includegraphics[width=4.00in]{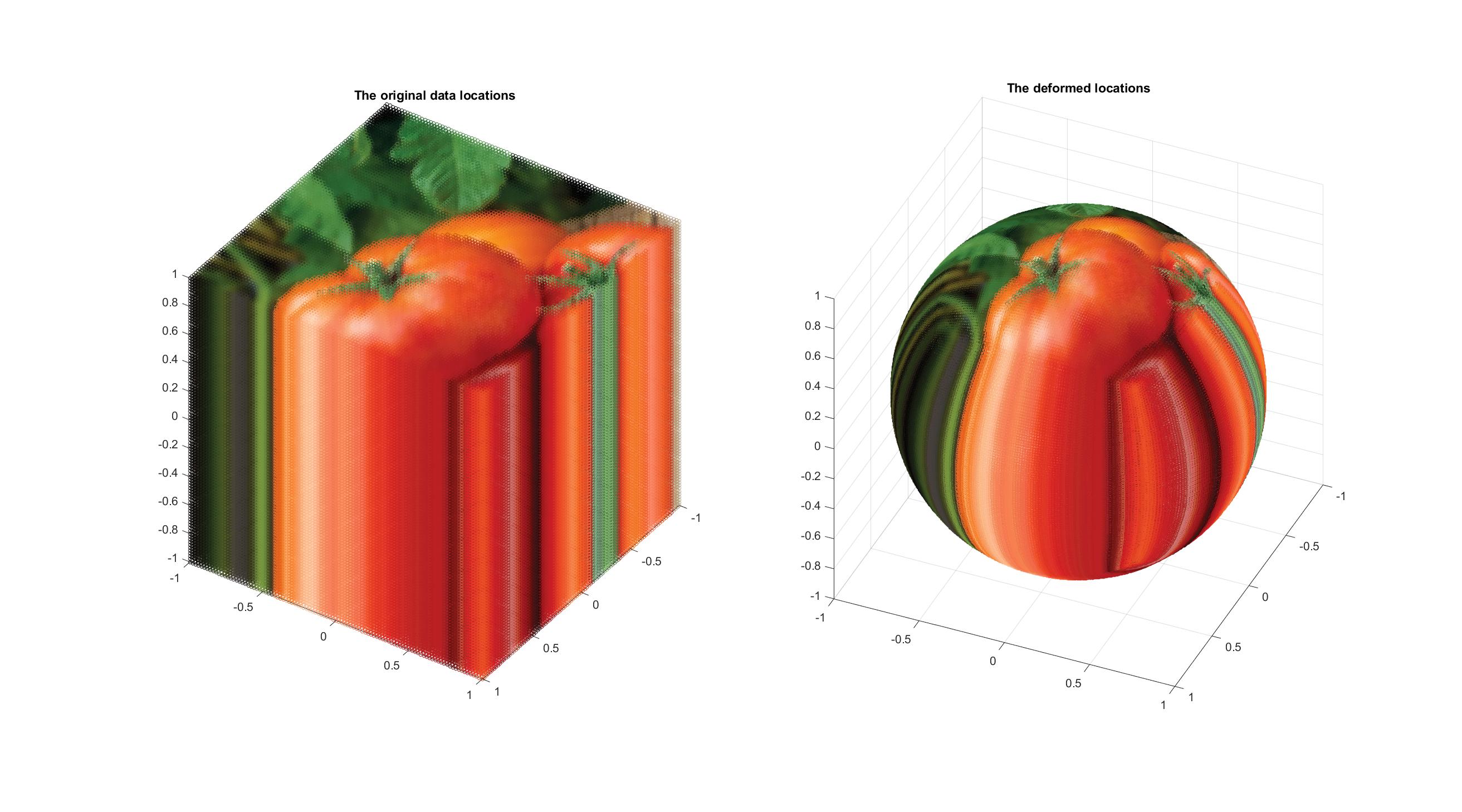}  
   \caption{A density function over the unit cube (left) is optimally transported to the ball (right)}
   \label{3dMAE}
\end{figure}

\begin{figure}[h]
         \includegraphics[width=5.00in]{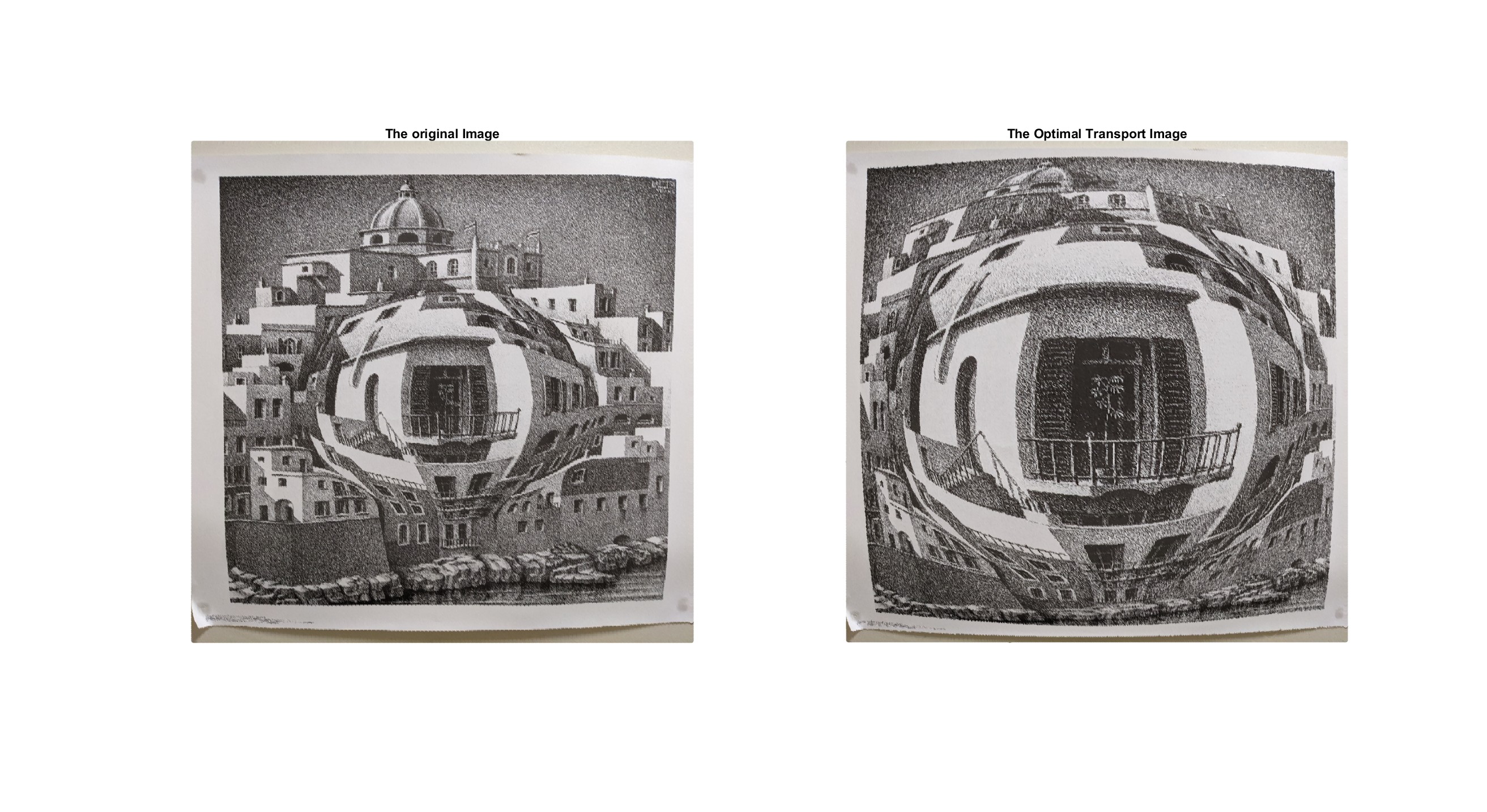}  
   \caption{An Escher art (left) and an another application of Escher's idea (right) based on bivariate spline solution of the Monge-Amp\'ere equation}
   \label{Escher}
\end{figure}
We leave the detail to \cite{LL24} for how to compute these imagte deformation.   

\section{Overcome the Curse of Dimensionality}
When approximating a high dimensional function, 
the computation of classic approximation methods suffers from the curse of dimensionality.  
For example, suppose that  $f\in C([0, 1]^d)$ with $d\gg 1$.
One usually uses Weierstrass theorem to have a degree $n$ polynomial $P_f$ such that 
$$
\|f- P_f\|_\infty \le \epsilon
$$ 
for any given tolerance $\epsilon>0$.
As the dimension of polynomial space $= {n+d\choose n}\approx n^d$ when $n >d$, one 
will need at least $N=O(n^d)$ data points in $[0, 1]^d$ to distinguish different polynomials in 
$\mathbb{P}_n$ and hence, to determine this $P_f$.
For example, 
it is known that for continuous functions $f\in C[0,1]$, Bernstein approximations 
$$
B_n(f)(x) = \sum_{i=0}^n f(\frac{i}{n}) b_{n,i}(x)
$$
where $b_{n,i}(x)= \frac{n!}{i! (n-i)!} x^i (1-x)^{n-i}$ can approximate $f$. 
\begin{lemma}
Suppose that $f\in C^2[0,1]$. 
\begin{equation}
|f(x)- B_n(f)| \le \frac{1}{8n} |f''|_{\infty}.
\end{equation}
\end{lemma}

To approximate a high dimensional function $f\in C[0, 1]^d$, 
write 
$$B_n(f) = \sum_{i_1,\cdots, i_d=0}^n f(\frac{i_1}{n}, 
\cdots, \frac{i_d}{n}) b_{n,i_1}(x_1) \cdots b_{n, i_d}(x_d).
$$ 
A standard tensor product approximation yields 
\begin{theorem}
Suppose that $f\in C^2[0,1]^d$. Then 
\begin{equation}
|f(x)- B_n(f)| \le \frac{d}{8n} |f''|_{\infty}.
\end{equation}
where 
$$B_n(f) = \sum_{i_1,\cdots, i_d=0}^n f(\frac{i_1}{n}, 
\cdots, \frac{i_d}{n}) b_{n,i_1}(x_1) \cdots b_{n, i_d}(x_d).
$$ 
\end{theorem}
Note that $B_n(f)$ requires $(n+1)^d$ terms.  This is the bottleneck of the approximation 
of high dimensional functions. 
A practical and theoretical question is: is it possible to sample fewer function data and use 
a few basis functions for 
a good approximation of high dimensional function?

It turns out that this is possible by using Kolmogorov Superposition Theorem (KST).
\begin{theorem}[Lorentz 1966\cite{L66}] 
\label{kolmogorov}
There exists irrational numbers $0<\lambda_p\leq 1$ for $p=1,\cdots, d$ and strictly increasing 
$Lip(\alpha)$ functions $\phi_q(x)$  
defined on $I=[0,1]$ for $q=0, \cdots, 2d$ such that 
for every $f\in C([0,1]^d)$, 
there exists a continuous function $g(u)$, $u\in [0,d]$, such that 
\begin{equation}
\label{repres}
    f(x_1,\cdots,x_d)=\sum_{q=0}^{2d} g(\sum_{i=1}^d\lambda_i\phi_q(x_i)).
\end{equation}
\end{theorem}
\begin{proof}
The proof is constructive. See \cite{L66} for detail. 
\end{proof}
\begin{remark}
$\phi_q$ are independent of $f$, but $g$ is dependent on $f$. 
For any $f\in C[0,1]^d$, there is a continuous function $g_f\in C[0, d]$ such that 
\begin{equation}
\label{representation}
    f(x_1,\cdots,x_d)=\sum_{q=0}^{2d} g_f(\sum_{i=1}^d\lambda_i\phi_q(x_i)).
\end{equation}
On the other hand, for any $g\in C[0, d]$, there is a continuous function $f_g\in C[0, 1]^d$ 
via the representation above. 
\begin{equation}
\label{representation2}
    f_g(x_1,\cdots,x_d)=\sum_{q=0}^{2d} g(\sum_{i=1}^d\lambda_i\phi_q(x_i)).
\end{equation}
\end{remark}

For example, if $g_n(t)=t^n$ is a polynomial of degree $n$ over $[0, d]$, then 
\begin{equation}
\label{Kpolynomial}
    Kp_n(x_1,\cdots,x_d)=\sum_{q=0}^{2d} (\sum_{i=1}^d\lambda_i\phi_q(x_i))^n
\end{equation}
is called K-polynomial of degree $n\ge 1$.  One can combine 
the Kolmogorov superposition theorem and the well known Weierstrass Theorem together to have
\begin{theorem}[K-Weierstrass Theorem (cf. \cite{LS23})]
For any $f\in C[0, 1]^d$, for any $\epsilon>0$, there exists a K-polynomial $Kp_n$ such that 
\begin{equation}
\label{Kpolynomial}
    |f(x_1, \cdots, x_d)- Kp_n(x_1, \cdots, x_d)|\le \epsilon,
\end{equation}
where $Kp_n$ has $(2d+1) (n+1)$ terms. 
\end{theorem}
\begin{proof}
Let $g_f$ be the K-outer function which is in $C[0, d]$. For $\epsilon/(2d+1)$, we use 
Weierstrass theorem to 
have a polynomial $p_n$ such that $|g_f(t)- p_n(t)|\le \epsilon/(2d+1)$.  It follows from 
Lorentz's representation (\ref{repres})  and (\ref{Kpolynomial}) that 
\begin{eqnarray}
|f(x_1, \cdots, x_d) - Kp_n(x_1,\cdots,x_d)| &\le & \sum_{q=0}^{2d}
|g_f(\sum_{i=1}^d\lambda_i\phi_q(x_i)) -p_n(\sum_{i=1}^d\lambda_i\phi_q(x_i))|\cr
&\le & \sum_{q=0}^{2d} \epsilon/(2d+1) = \epsilon.
\end{eqnarray}  
Hence, this completes the proof.  
\end{proof}

We now expplain how to use B-splines to approximate $g_f$.  
Let $B_{n,i}, i=1, \cdots, nd$ be the B-spline of degree $k$  
with $nd$ equally-spaced knots over $[0, d]$ as K-outer function $g$ to generate 
\begin{equation}
\label{KBsplines}
KB_{n,i}(x_1, \cdots, x_d) = \sum_{q=0}^{2d} B_{n,i}(\sum_{i=j}^d\lambda_j\phi_q(x_j))
\end{equation}
which are called KB splines of degree $k$ for short. Some researchers called those functions  
Kolmogorov spline networks (see, e.g. \cite{IP03}).     

\begin{theorem}[Lai and Shen, 2023\cite{LS22}]
\label{rate3}
Suppose that  $f\in C([0,1]^d)$. Then 
there is a KB spline $KB_{n}(f)=\sum_{i=1}^{nd} c_i(f) KB_{n,i}$ such that 
\begin{equation}
\label{linearrate}
|f(x_1,\cdots, x_d)- KB_{n}(f)(x_1, \cdots,x_d)|\le (2d+1)\omega(g_f, 1/n),
\end{equation}
for all $\bfx=(x_1, \cdots, x_d)\in [0, 1]^d$, where 
$\omega(g_f,t)$ is the modulus of continuity of K-outer function $g_f$.
\end{theorem}

However, it is hard to implement $\phi_q$'s accurately, in particular when $d\gg 1$. Hence,  
these KB-splines are very noisy which can be seen in Figure~\ref{LKB}.  
A few  researchers thought that they are useless   
as the noises are inherent and essential,  e.g. \cite{GP89}.

One approach discovered in \cite{LS22} is to smooth these KB-splines so that they can be useful. 
Multivariate splines have been used for scattered data fitting and denoising by using the
so-called penalized least squares method. See, e.g. \cite{LW13}.  The classic penalized 
least squares splines method can be recalled as follows. 
 For a given data set $\{(x_i,y_i, z_i), i=1, \cdots, N\}$ with 
$(x_i,y_i)\in [0, 1]^2$ and $z_i= f(x_i,y_i)+ \epsilon_i, i=1, \cdots, N$  
with noises $\epsilon_i$ which may not be 
very small, the penalized least squares method  is to find 
\begin{equation}
\label{PLS}
\min_{s\in S^1_5(\triangle)} \sum_{i=1, \cdots, N} |s(x_i,y_i)- z_i|^2  
+ \lambda {\cal E}_2(s) 
\end{equation}
with $\lambda\approx 1$, where ${\cal E}_2(s)$ is the thin-plate energy functional 
defined as follows.
\begin{equation}
\label{E2}
{\cal E}_2(s) =\int_\Omega |\frac{\partial^2}{\partial x^2} s|^2 + 2
|\frac{\partial^2}{\partial x \partial y} s|^2 + |\frac{\partial^2}{\partial y^2} s|^2. 
\end{equation} 
The researchers in \cite{LS22} adopted this approach with $\lambda=O(1)$ as the noises are large.    
It turns out that the multivariate spline denoising method is very effective. 
The smooth version of KB-splines are called LKB splines. Some of them are shown in Figure~\ref{LKB}. 
These splines are in $S^2_8(\triangle_{32})$ with 32 triangles and 25 vertices.  The choice of this 
spline space is a compremise between the efficiency and effectiveness of the denoising.
In fact, some other spline spaces, e.g. $S^2_8(\triangle_{128})$ and $S^2_{12}(\triangle_{32})$ 
can do denoising better, but only slightly better with a significantly more 
computational time than $S^2_8(\triangle_{32})$.  
 
\begin{figure}[h]
    \centering
    \begin{tabular}{cc} 
        \includegraphics[width=2.00in]{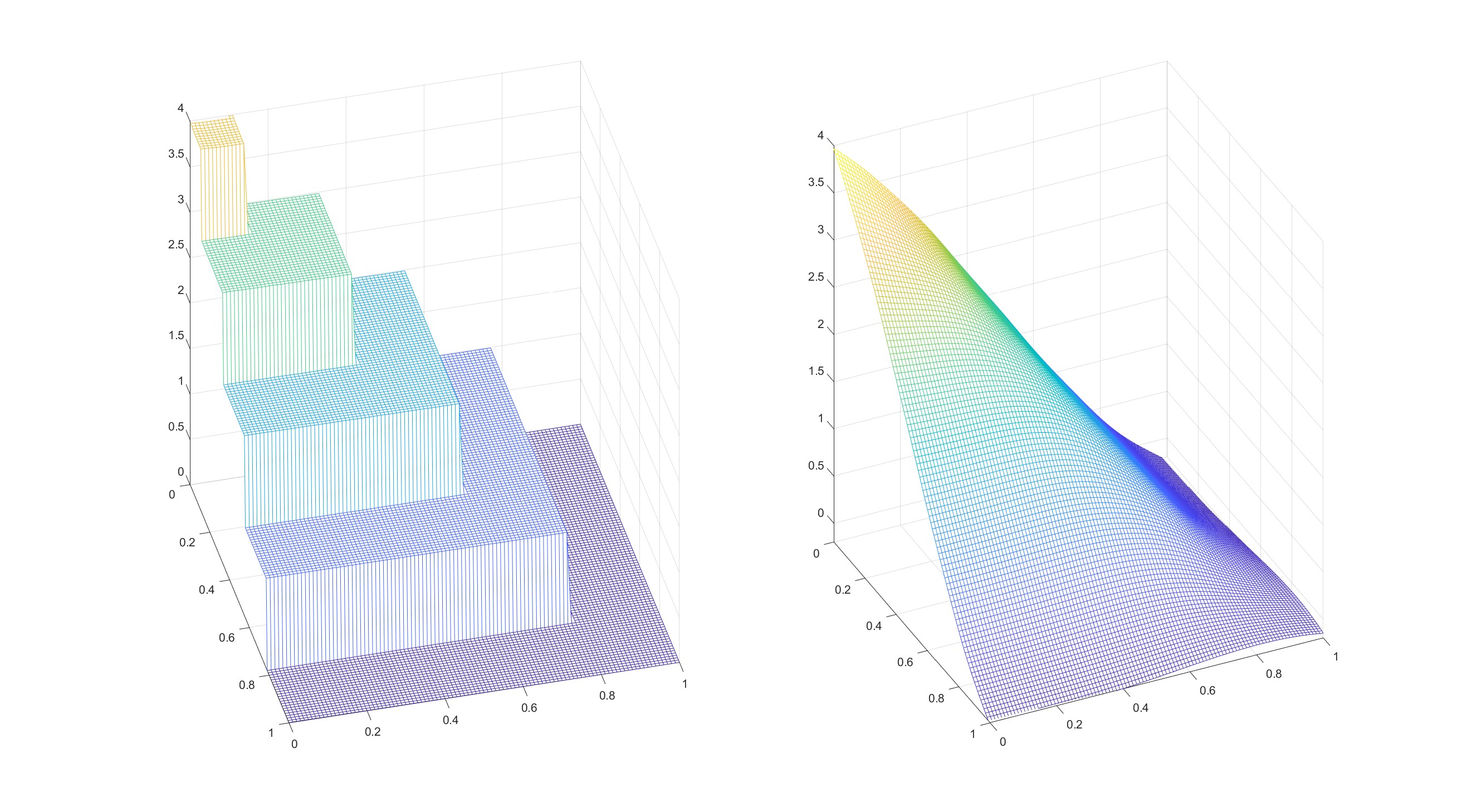}  &  \includegraphics[width=2.00in]{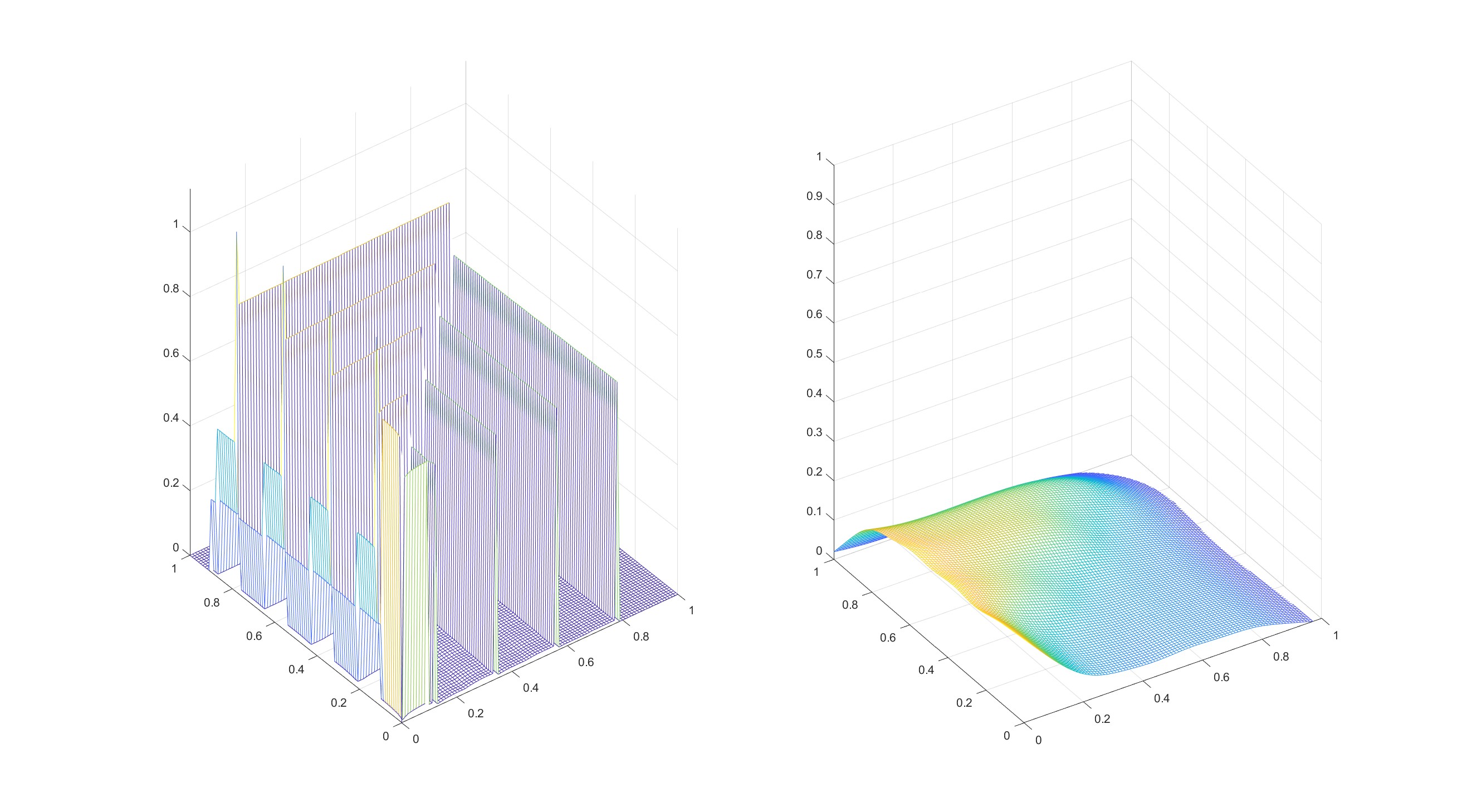} \\
        \includegraphics[width=2.00in]{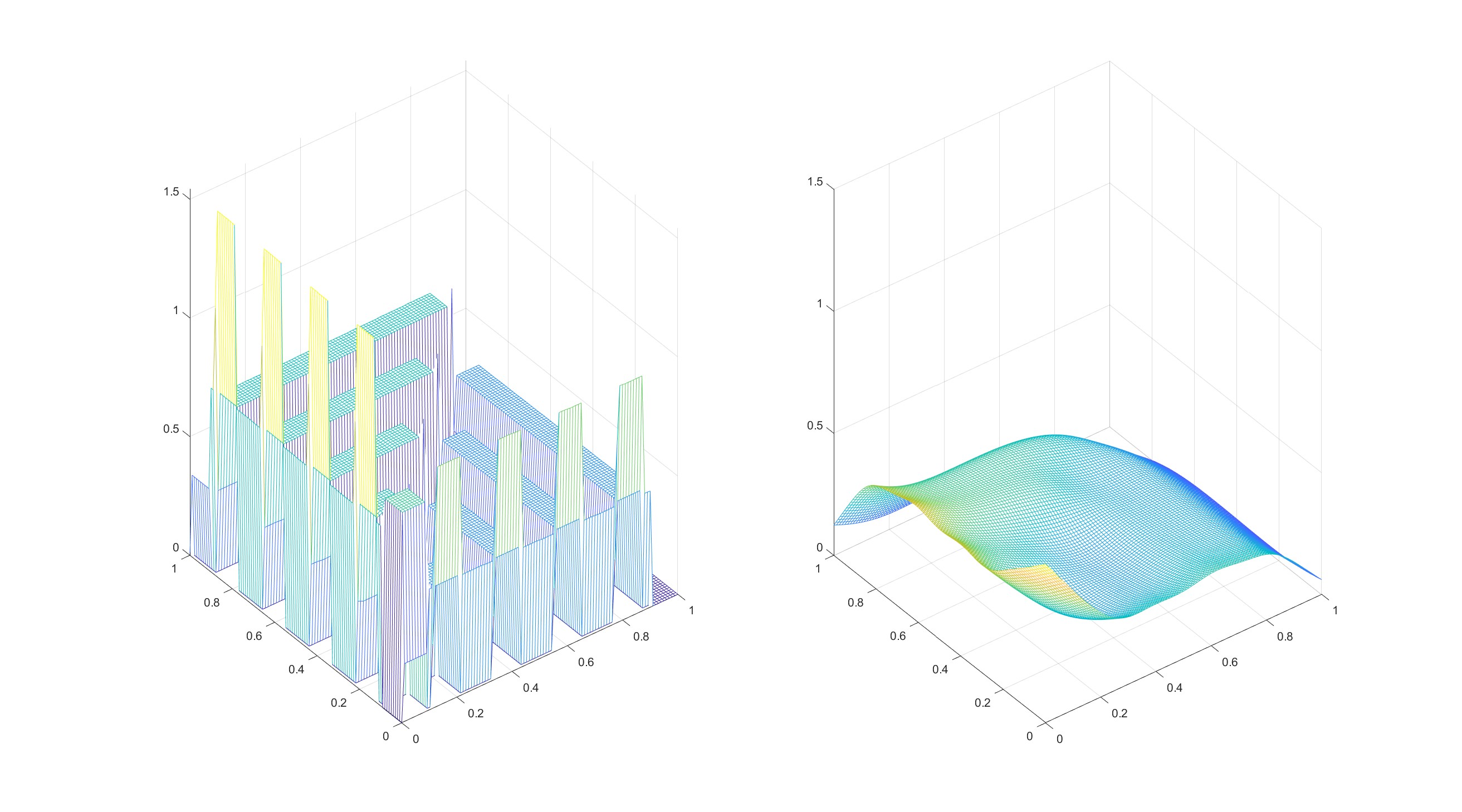} &
        \includegraphics[width=2.00in]{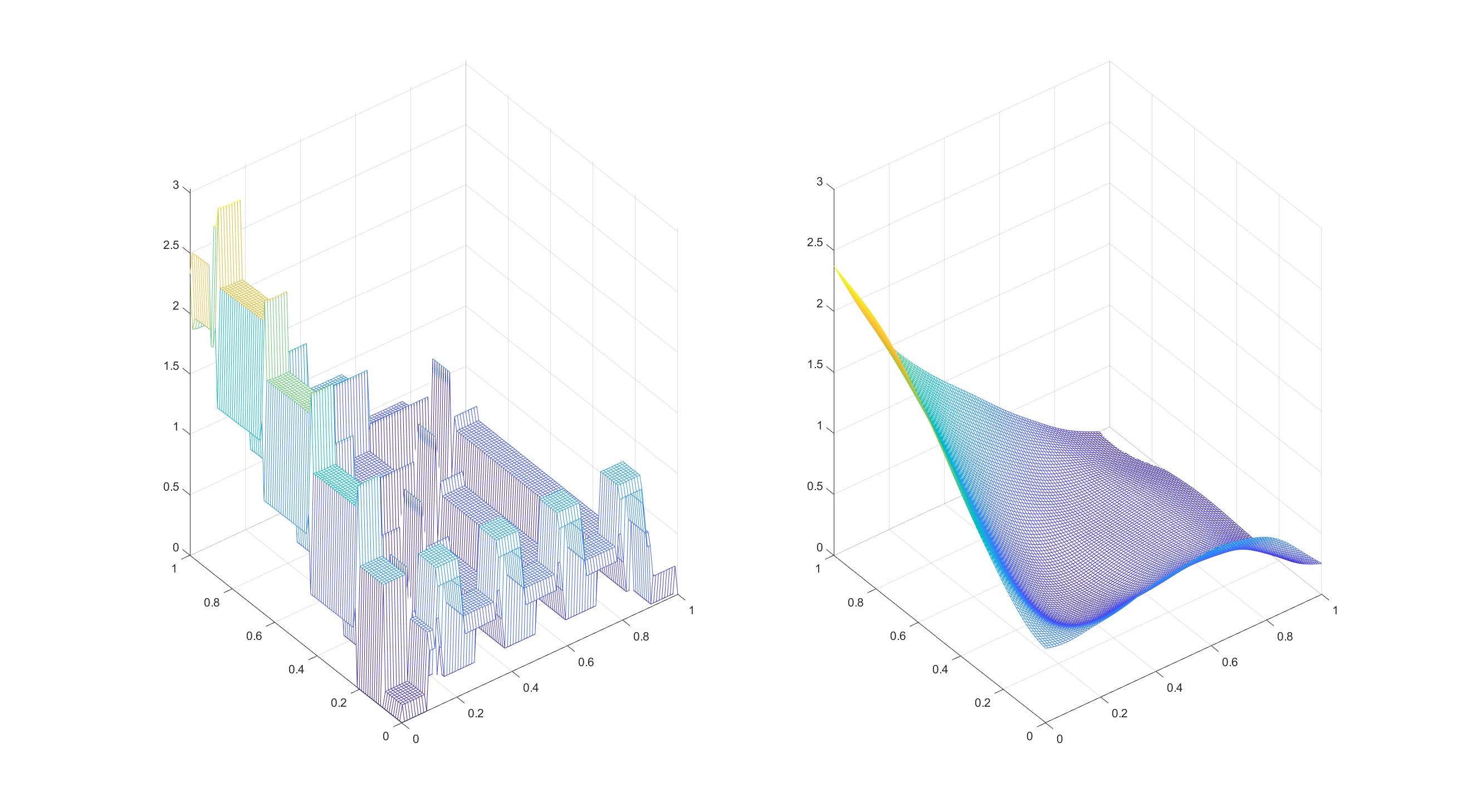} \\
              \includegraphics[width=2.00in]{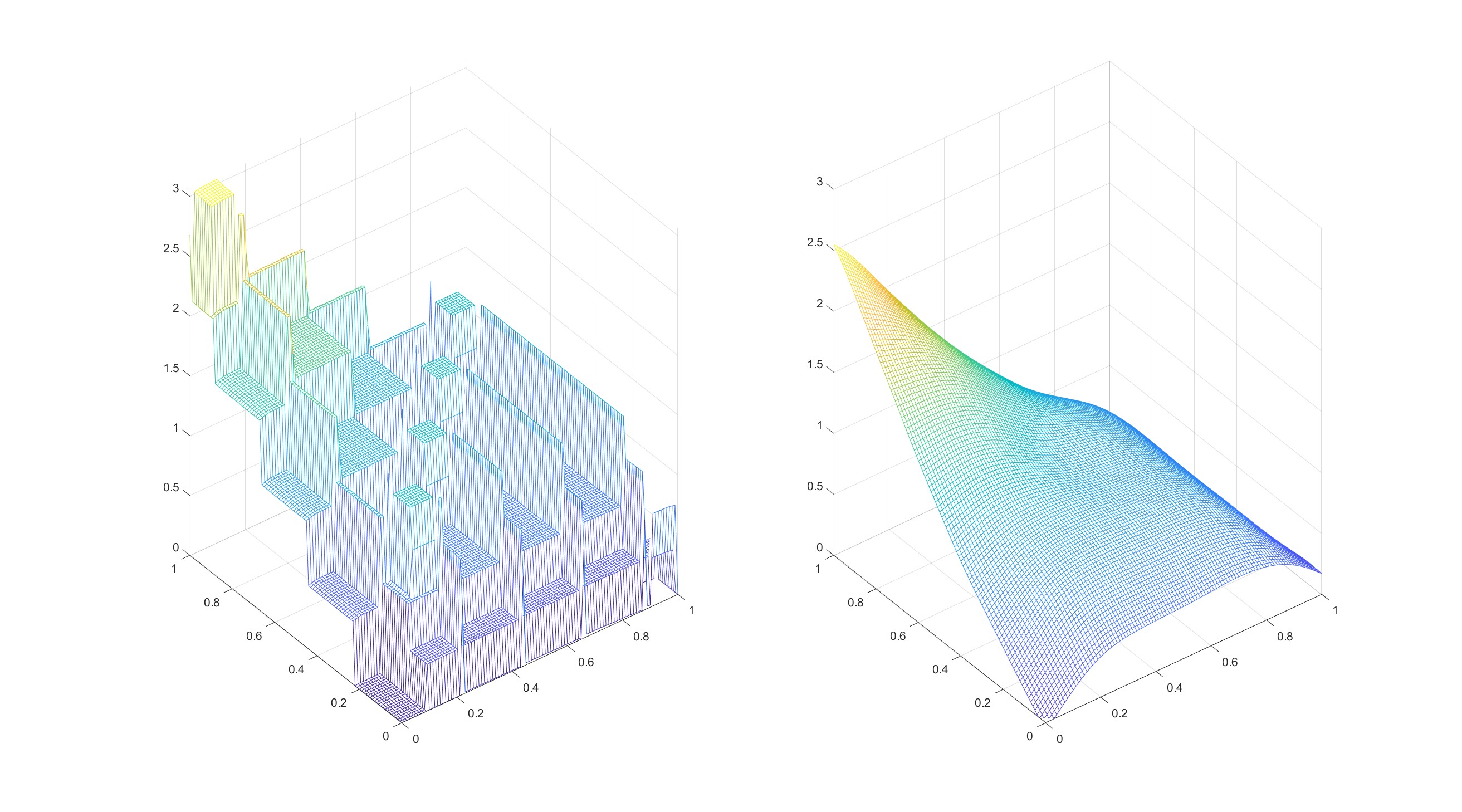}  &  \includegraphics[width=2.00in]{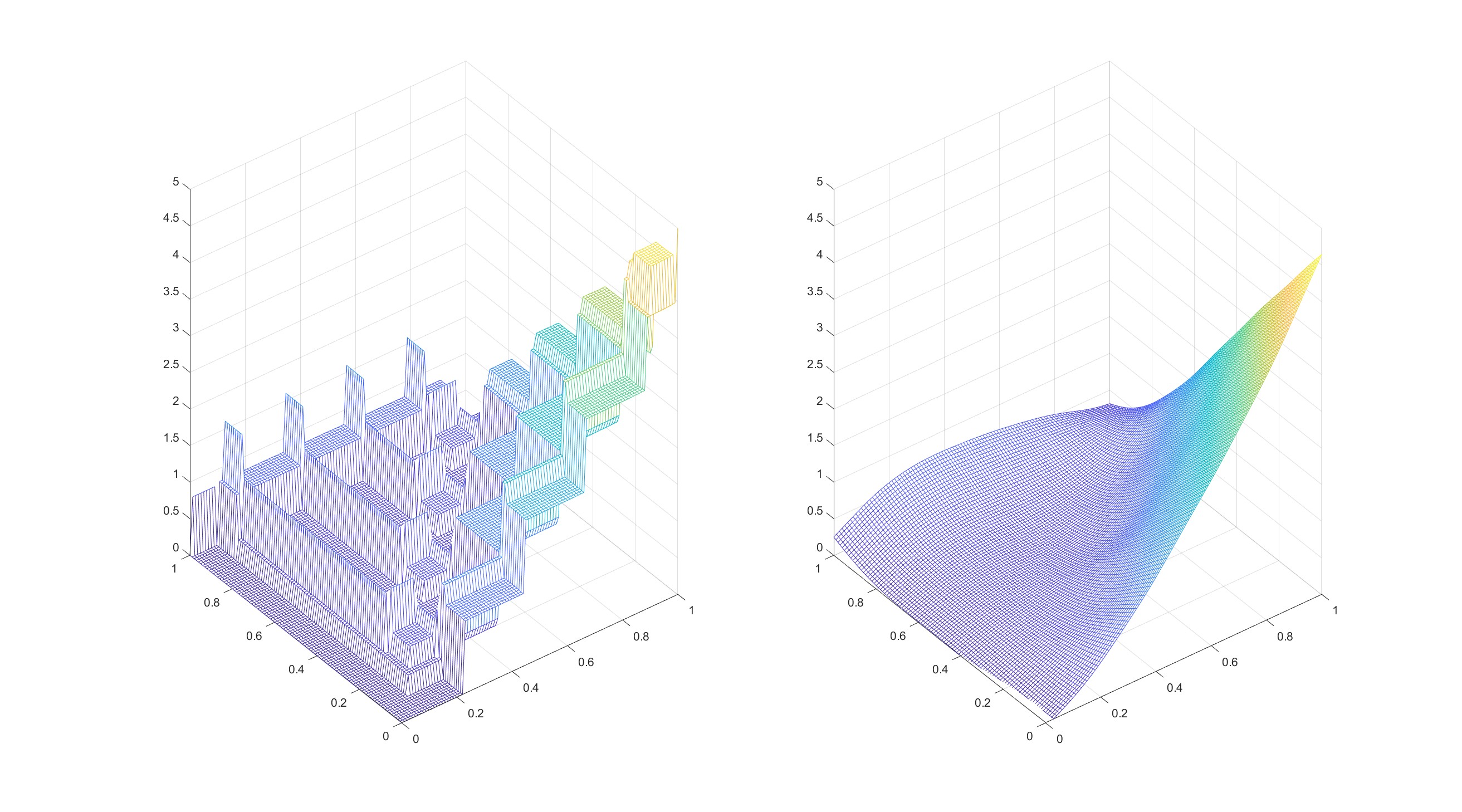} \\
               \includegraphics[width=2.00in]{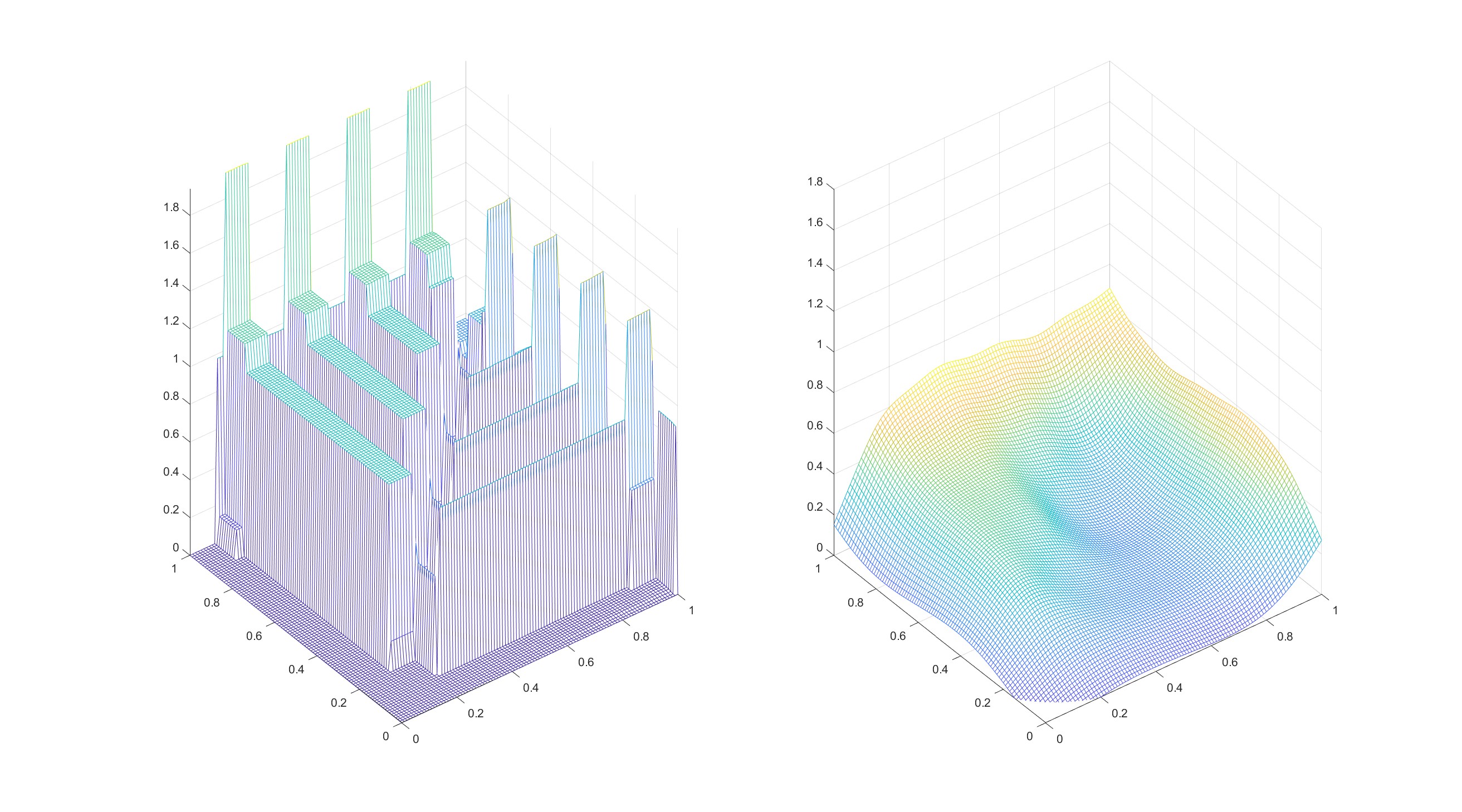}  &  \includegraphics[width=2.00in]{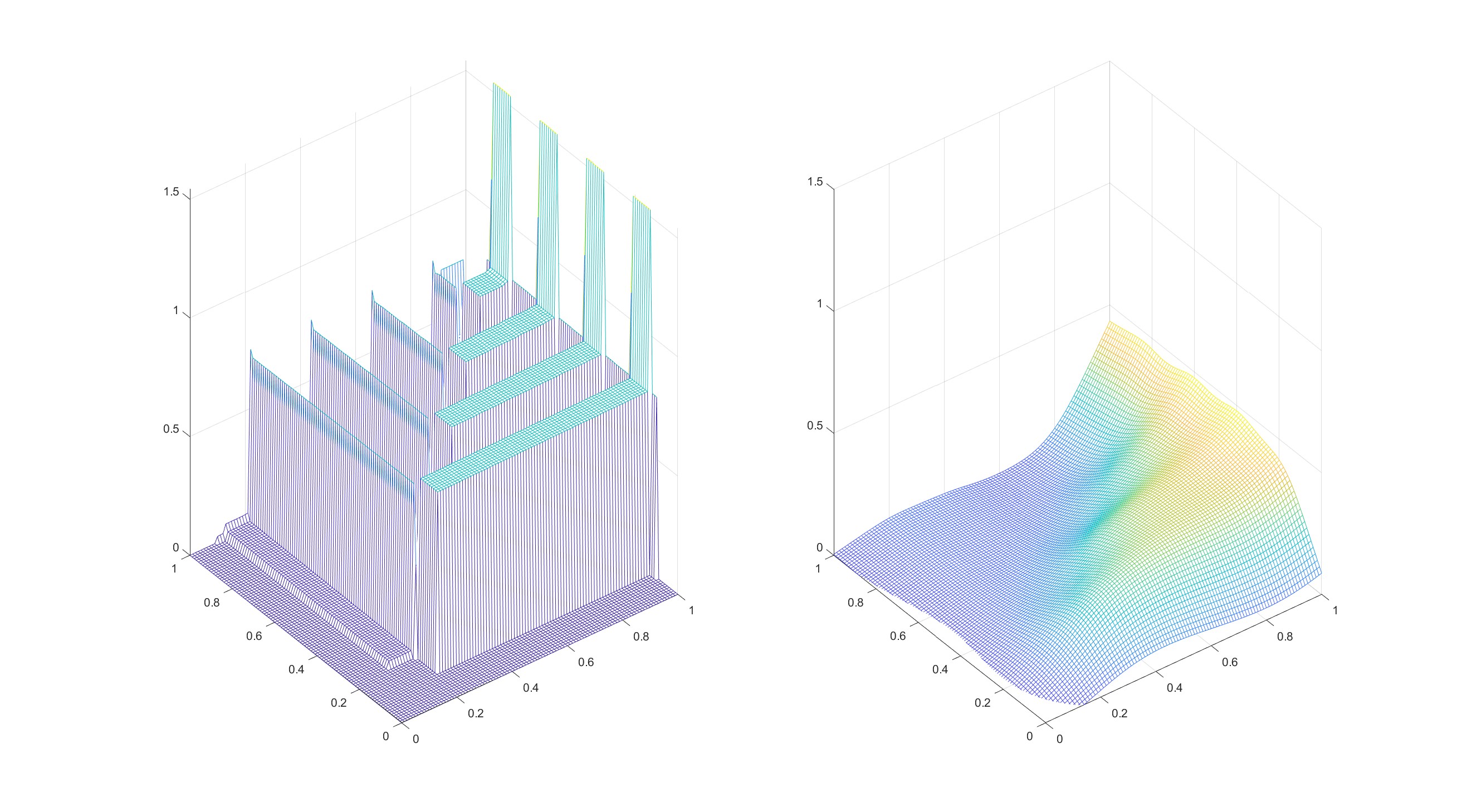} \\
                \includegraphics[width=2.00in]{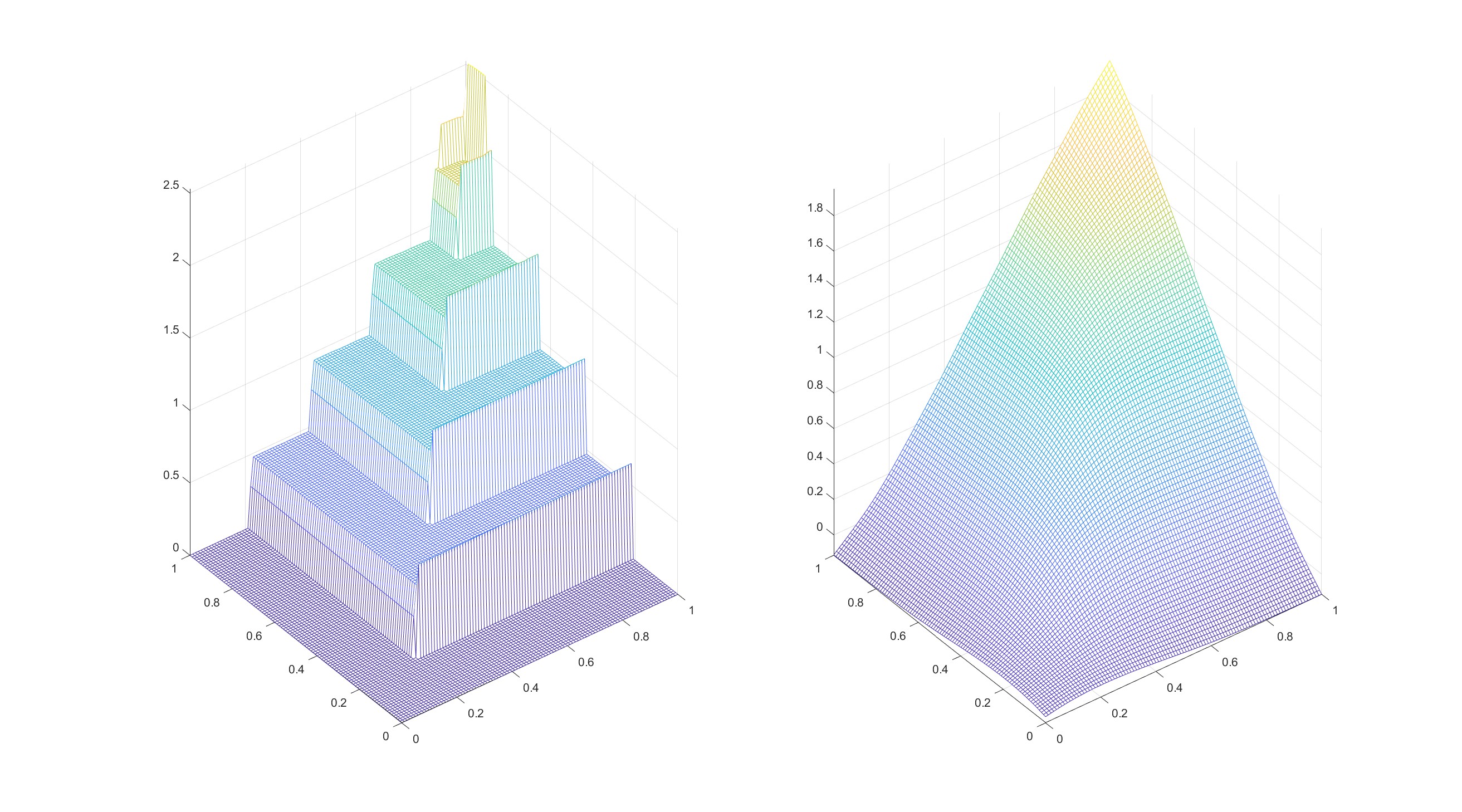}  &  \includegraphics[width=2.00in]{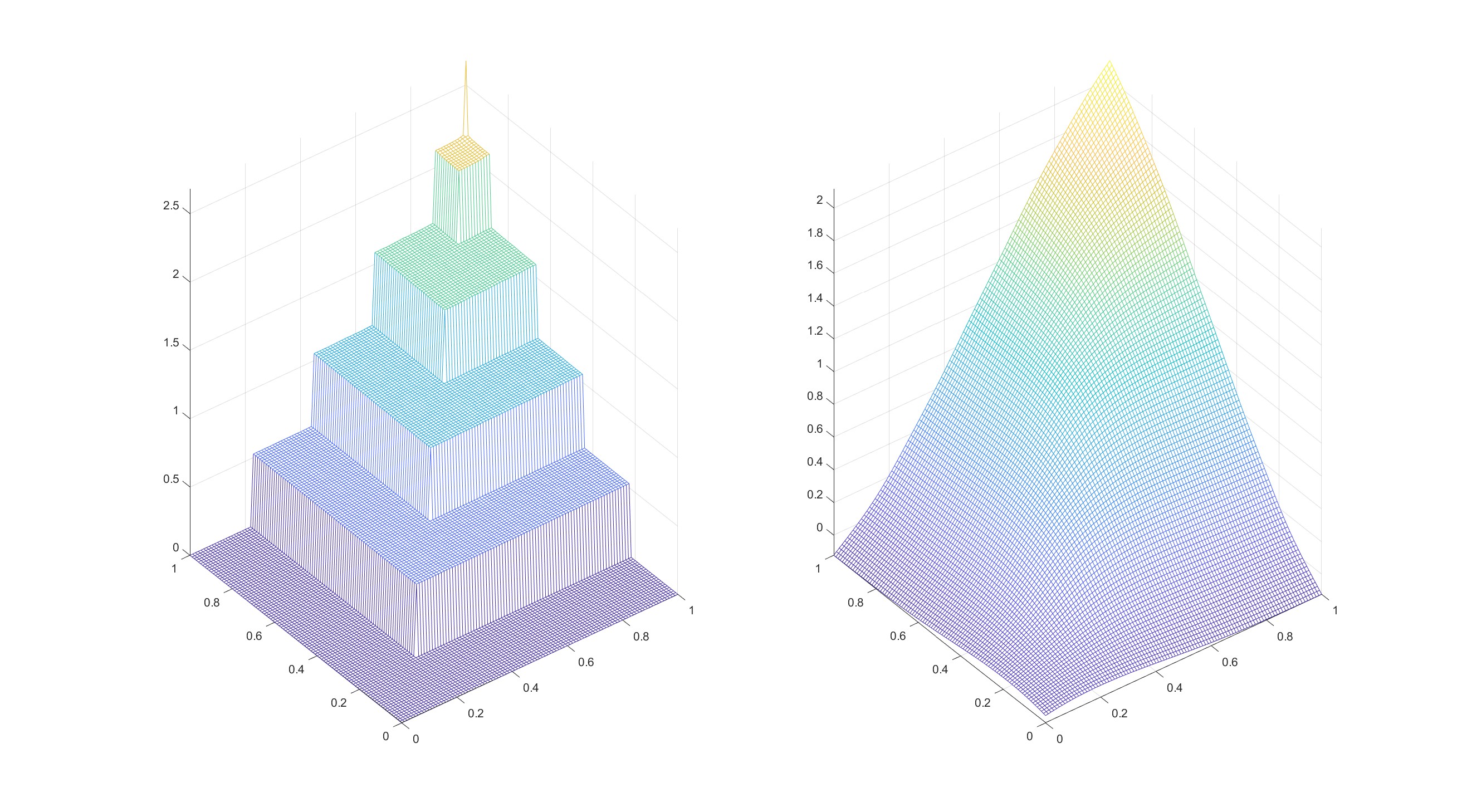} \\
    \end{tabular}
    \caption{\small{Some Examples of LKB-splines (the second and fourth columns) which are the smoothed version of the corresponding 
    KB-splines (the first and third columns)}. \label{LKB}}
\end{figure}

In \cite{LS22}, the researchers show that these LKB splines are indeed useful. 
In the 2D setting, they uniformly sampled $101^2$ across $[0,1]^2$ and fit the discrete least squares (DLS) approximation of a continuous function $f$ with LKB-splines. 
The following 10 testing functions across different families of continuous functions are 
used to check the approximation accuracy. 
\begin{eqnarray*} 
  f_1&=& x^2; f_2= xy; f_3= \sin(x); \cr 
f_{4}&=& \tan(x-y)/\tan(1);\cr 
 f_{5}&=& \sin(\sin(\sin(\sin(x^2-y^2)))) ;\cr 
  f_{6}&=& \exp(1-(x-0.5)^2-(y-0.5)^2)/\exp(1);\cr
 f_{7}&=& \log(1+x^2+y^2)/\log(4);\cr
 f_{8}&=&(x+2y)/(3(1+y^2+x^2));\cr
 f_{9}&=& \tan(x^2-y^2)/\tan(1);\cr
 f_{10}&=& \exp(-\cos(1+\sin(1+\cos(x^2-y^2))))/\exp(1);\cr
\end{eqnarray*}

RMSEs in Table~\ref{2Dex1} are computed based on $1001\times 1001$ equally-spaced points over $[0, 1]^2$ to 
check the errors. Some numerical results in 2D setting are reported in this paper. 
\begin{table}[htpb]
\caption{RMSEs of the DLS Fitting of 10 Testing Functions Using $2n$ LKB-splines in 2D with $101^2$ 
Sampled Data. \label{2Dex1}}
\centering
\begin{tabular}{|c|c|c|c|c|}\hline 
Testing Funs & $n=10$ & $n=100$ & $n=1000$ & $n=10000$ \cr \hline  
   $f_1$& 1.126e-02  & 2.610e-04 &  1.200e-04 & 1.248e-07 \cr \hline  
   $f_2$&  2.865e-03 &  7.269e-04 & 2.590e-05 &1.577e-08 \cr \hline  
  $f_3$& 3.609e-03  & 7.526e-05 &  3.407e-05 & 4.203e-08 \cr \hline  
  $f_4$&   2.969e-03 &  2.974e-04 &  1.453e-04 & 2.460e-07 \cr \hline  
  $f_{5}$&   1.886e-02&   7.694e-04 &  3.071e-04 &2.106e-07 \cr \hline  
 $f_{6}$&  4.513e-03  & 1.828e-04 &  6.697e-05 & 6.484e-08 \cr \hline  
  $f_{7}$&   1.957e-03 &  8.415e-05 &  4.009e-05 & 6.345e-08 \cr \hline  
  $f_{8}$&   2.685e-03 &  1.065e-04 &  3.946e-05 & 2.434e-08 \cr \hline  
  $f_{9}$&  2.058e-02  & 1.333e-03 &  6.295e-04 & 1.096e-06 \cr \hline  
  $f_{10}$&   1.502e-03 &  8.060e-05 &  3.472e-05 & 5.492e-08 \cr \hline  
   \end{tabular}
   \end{table}  
 Much more results can be found in \cite{LS23}. These show that the denoising by using multivariate spline  functions works very well. Furthermore, the researchers in \cite{LS22} used the pivotal point set 
 for approximating 2D and 3D functions. With the function values over a set of $2n$ pivotal points 
 in $[0, 1]^2$, their $2n$ LKB splines can approximate many continuous functions very well with the 
 convergence rate $O(1/n)$ if the functions are K-Lipschitz continuous.  Similar for the 3D, 4D, 6D settings. 
See \cite{LS22} and \cite{LS23}. These show that  
 the curse of dimensionality can be overcome when using LKB splines.

\section{Conclusions and Remarks}
This paper is dedicated to Professor Larry L. Schumaker in recognition of his lifelong dedication and significant contributions to the theory and applications of multivariate splines. His work has been immensely fruitful, leading to the development of spline functions that are not only extremely useful but also unexpectedly effective in overcoming the challenges of high-dimensional function approximation. Specifically, the LKB splines, referenced in the preceding section, have shown great potential in applications such as numerical solutions to the Poisson equation (see \cite{LS23}) and in numerical quadrature techniques.

\bmhead{Acknowledgments}
The authors are very grateful to the editor and referees for their helpful comments.

\section*{Declarations}


\begin{itemize}
\item \textbf{Funding.} The author is supported by the Simons Foundation for collaboration grant \#864439. 
\item \textbf{Competing interests.} The  author states that there is no conflict of interest.
\end{itemize}

\end{document}